\documentclass[a4paper, 11pt, leqno, dvipsnames]{amsart}

\usepackage[french, english]{babel} 
\usepackage[T1]{fontenc}
\usepackage[ansinew]{inputenc}

\date{}

\title[Approximate null-controllability for the Ornstein-Uhlenbeck equations]{Approximate null-controllability with uniform cost for the hypoelliptic Ornstein-Uhlenbeck equations}

\author{Paul Alphonse}
\address{(Paul Alphonse) Université de Lyon, ENSL, UMPA - UMR 5669, F-69364 Lyon}
\email{paul.alphonse@ens-lyon.fr}

\author{Jérémy Martin}
\address{(Jérémy Martin) Univ Rennes, CNRS, IRMAR - UMR 6625, F-35000 Rennes}
\email{jeremy.martin@univ-rennes1.fr}

\keywords{approximate null-controllability; integral thickness condition; unique continuation property; Ornstein-Uhlenbeck equations; hypoellipticity}

\makeatletter
	\@namedef{subjclassname@2020}{\textup{2020} Mathematics Subject Classification}
\makeatother

\subjclass[2020]{93B05, 35H10, 47D06, 42B37}

\usepackage[top=3cm, bottom=2cm, left=3cm, right=3cm]{geometry}	

\usepackage{enumitem}

\usepackage{array}											

\usepackage{amsfonts}
\usepackage{amsmath}
\usepackage{amsthm}
\usepackage{amssymb}
\usepackage{mathtools}

\usepackage{multicol}

\usepackage{bbm}

\usepackage{stmaryrd}

\usepackage[scr]{rsfso}

\usepackage{comment}

\usepackage[dvipsnames]{xcolor}

\usepackage{esint}

\usepackage{ulem}

\usepackage{hyperref}

\hypersetup{	
colorlinks=true,
breaklinks=true,
urlcolor= red,
linkcolor= red,
citecolor = Blue
}

\numberwithin{equation}{section}

\newtheorem{thm}{Theorem}[section]
\newtheorem{prop}[thm]{Proposition}

\newtheorem{lem}[thm]{Lemma}
\newtheorem{cor}[thm]{Corollary}
\theoremstyle{definition}

\newtheorem{ex}[thm]{Example}
\newtheorem{rk}[thm]{Remark}

\DeclareMathOperator{\Tr}{Tr}
\DeclareMathOperator{\Leb}{Leb}

\DeclareMathOperator{\Rank}{Rank}
\DeclareMathOperator{\Ker}{Ker}

\begin{document}

\sloppy

\selectlanguage{english}

\begin{abstract} We prove that the approximate null-controllability with uniform cost of the hypoelliptic Ornstein-Uhlenbeck equations posed on $\mathbb R^n$ is characterized by an integral thickness geometric condition on the control supports. We also provide associated quantitative weak observability estimates. This result for the hypoelliptic Ornstein-Uhlenbeck equations is deduced from the same study for a large class of non-autonomous elliptic equations from moving control supports. We generalize in particular results known for parabolic equations posed on $\mathbb R^n$, for which the approximate null-controllability with uniform cost is ensured by the notion of thickness, which is stronger than the integral thickness condition considered in the present work. Examples of those parabolic equations are the fractional heat equations associated with the operator $(-\Delta)^s$, in the regime $s\geq1/2$. Our strategy also allows to characterize the approximate null-controllability with uniform cost from moving control supports for this class of fractional heat equations.
\end{abstract}

\maketitle

\section{Introduction}

The study of the (rapid) stabilization and the (approximate) null-controllability of parabolic equations \cite{AM, MR3816981, HWW, Ko, L, NTTV, WZ} or degenerate parabolic equations of hypoelliptic type \cite{AB, BEPS, BJKPS, MR3680980, DSV} posed on $\mathbb R^n$ and taking the following form
\begin{equation}\label{25052018E3}\tag{$E_P$}
\left\{\begin{array}{l}
	\partial_tf(t,x)+Pf(t,x) = h(t,x)\mathbbm1_{\omega}(x),\quad (t,x)\in(0,+\infty)\times\mathbb R^n,\\[5pt]
	f(0,\cdot) = f_0\in L^2(\mathbb R^n),
\end{array}\right.
\end{equation}
has been much addressed recently. The purpose of this line of research is to provide \textit{geometric} characterizations for the control support $\omega\subset\mathbb R^n$ that ensure the above notions for the equations \eqref{25052018E3}. At the present time, the stabilization and the null-controllability properties are well-understood for a large class of parabolic equations posed on $\mathbb R^n$, as we will detail just after. The case is similar for the parabolic equations posed on bounded domains, as for the heat equation whose null-controllability properties are known for decades \cite{MR1312710} and whose stabilization properties have been recently investigated \cite{ Xiang}. However, the situation is different for the hypoelliptic equations of the form \eqref{25052018E3}, whose study is only at an early stage. For this class of equations, we currently do not have any necessary and sufficient geometric characterization on $\omega\subset\mathbb R^n$ that ensures their null-controllability, even on particular examples. The hypoelliptic equations posed on bounded domains or on manifolds are also widely studied, and the situation is quite different for them. Although these equations have not been studied in a general setting, some particular examples as the Grushin equation, the Kolmogorov equation or the heat equation on the Heisenberg group are now quite well-understood \cite{zbMATH06680964, zbMATH07206863, zbMATH07352607, LL}.

In this work, we study the \textit{cost-uniformly approximate null-controllability} properties of the equation \eqref{25052018E3} associated with the following \textit{hypoelliptic} Ornstein-Uhlenbeck operator
\begin{equation}\label{12052021E1}
	P = QD_x\cdot D_x + Bx\cdot\nabla_x,\quad x\in\mathbb R^n,
\end{equation}
where $B$ and $Q$ are $n\times n$ real matrices, $Q$ being moreover symmetric positive semidefinite. Let us recall that the hypoellipticity of the operator $P$ is characterized by a simple algebraic condition on the matrices $B$ and $Q$ known as the Kalman rank condition \eqref{05072021E3} presented shortly after. Precisely, we prove that for all positive time $T>0$, the evolution equation \eqref{25052018E3} is cost-uniformly approximately null-controllable from the control support $\omega$ in time $T$ \textit{if and only if} there exist a radius $r>0$ and a rate $\gamma\in(0,1]$ such that
\begin{equation}\label{20122021E1}
	\forall x\in\mathbb R^n,\quad\frac1T\int_0^T\Leb\big((e^{tB}\omega)\cap B(x,r)\big)\, \mathrm dt\geq\gamma V_r,
\end{equation}
where $V_r$ stands for the volume of a Euclidean ball of radius $r$ in $\mathbb R^n$. This above geometric condition will be called \textit{integral thickness condition}, since it generalizes the notion of thickness \eqref{07122021E4} which corresponds to the case where the matrix $B$ is zero, that is, to the elliptic case. This notion of thickness has turned out to play a key role in the theories of stabilization and (approximate) null-controllability, since it was proven to be a necessary and sufficient geometric condition that ensures these notions for large classes of parabolic equations posed on $\mathbb R^n$, as the fractional heat equations for instance, see e.g. \cite{AB, AM, MR3816981, HWW, LWXY, MR4041279, WZ}. The thickness condition is also involved in the study of the exact controllability of the free Schr\"odinger equation, as highlighted in \cite{HWWII, MPS2}. Moreover, the geometric condition \eqref{20122021E1} has already been introduced in \cite{BEPS}, where the authors proved that this is a necessary condition for the exact null-controllability of the hypoelliptic Ornstein-Uhlenbeck equation \eqref{25052018E3}. They actually consider a quite more general class of equations, which will be presented later. As a consequence of their work, we know that many hypoelliptic equations, as the Kolmogorov equation or the Kolmogorov equation with a quadratic external force, require a \textit{minimal time} to be possibly exactly null-controllable from specific control supports, as cones for instance. In the present work, we check that these minimal times are also required to obtain positive results of approximate null-controllability with uniform cost for the very same equations.

In fact, the result of approximate null-controllability with uniform cost obtained in this work for the hypoelliptic Ornstein-Uhlenbeck equation \eqref{25052018E3} is deduced from the study of the same notion for \textit{non-autonomous diffusive equations} of the form
\begin{equation}\label{05072021E1}\tag{$E_{Q_t}$}
\left\{\begin{aligned}
	& \partial_tf(t,x) + Q_tD_x\cdot D_xf(t,x) = h(t,x)\mathbbm 1_{\omega(t)}(x),\quad (t,x)\in(0,T]\times\mathbb R^n, \\
	& f(0,\cdot) = f_0\in L^2(\mathbb R^n),
\end{aligned}\right.
\end{equation}
where $(Q_t)_{t\in\mathbb R}$ is a family of real $n\times n$ matrices and $(\omega(t))_{t\in[0,T]}$ is a moving control support. This class of equations has also been considered in the work \cite{BEPS}, where the authors investigate their exact null-controllability (again, they consider a larger class of equations). Under an ellipticity assumption on the matrices $Q_t$, we prove that the equation \eqref{05072021E1} is cost-uniformly approximately null-controllable in time $T$ from $(\omega(t))_{t\in[0,T]}$ if and only if there exist a radius $r>0$ and a rate $\gamma\in(0,1]$ such that
\begin{equation}\label{20122021E2}
	\forall x\in\mathbb R^n,\quad\frac1T\int_0^T\Leb\big(\omega(t)\cap B(x,r)\big)\, \mathrm dt\geq\gamma V_r.
\end{equation}
As before for the equation \eqref{25052018E3}, the above geometric condition was proven in \cite{BEPS} to be necessary for the exact null-controllability of the non-autonomous diffusive equations \eqref{05072021E1}. 

The strategy of proof implemented in the present paper also allows to consider fractional diffusive models. More precisely, by adapting the study of the equation \eqref{05072021E1}, we get that the geometric condition \eqref{20122021E2} is necessary and sufficient to obtain a positive result of approximate null-controllability with uniform cost for the \textit{fractional heat equation} from moving control supports in a high diffusion setting, that is, for the evolution equation posed on $\mathbb R^n$ and associated with the operator $(-\Delta)^s$, in the regime $s\geq1/2$. This generalizes in particular our previous result \cite[Example 2.8]{AM}.

\subsubsection*{Notations} The following notations and conventions will be used all over this work:
\begin{enumerate}[label=\textbf{\arabic*.},leftmargin=* ,parsep=2pt,itemsep=0pt,topsep=2pt]
	\item The canonical Euclidean scalar product of $\mathbb R^n$ is denoted by $\cdot$ and $\vert\cdot\vert$ stands for the associated canonical Euclidean norm.
	\item For all measurable subset $\omega\subset\mathbb R^n$, the inner product of $L^2(\omega)$ is defined by
$$\langle u,v\rangle_{L^2(\omega)} = \int_{\omega}u(x)\overline{v(x)}\,\mathrm dx,\quad u,v\in L^2(\omega),$$
while $\Vert\cdot\Vert_{L^2(\omega)}$ stands for the associated norm.
	\item For all function $u\in\mathscr S(\mathbb R^n)$, the Fourier transform of $u$ is denoted $\widehat u$ or $\mathscr F u$ and is defined by
$$\widehat u(\xi) = (\mathscr Fu)(\xi) = \int_{\mathbb R^n}e^{-ix\cdot\xi}u(x)\,\mathrm dx,\quad \xi\in\mathbb R^n.$$
With this convention, Plancherel's theorem states that 
$$\forall u\in L^2(\mathbb R^n),\quad \Vert\widehat u\Vert_{L^2(\mathbb R^n)} = (2\pi)^{n/2}\Vert u\Vert_{L^2(\mathbb R^n)}.$$
	\item We denote by $\nabla_x$ the gradient and we set $D_x = -i\nabla_x$.
	\item For all measurable subset $\omega\subset\mathbb R^n$, $\mathbbm1_{\omega}$ denotes the characteristic function of $\omega$.
	\item Given some $r>0$, the notation $V_r$ stands for the volume of a Euclidean ball of radius $r$ in $\mathbb R^n$ with respect to the Lebesgue measure, which is denoted $\Leb$.
\end{enumerate}

\section{Statement of the main results}\label{results}

This section is devoted to present in details the main results contained in this work. Before stating those results, given some positive time $T>0$, let us define the different concepts related to the control system we are interested in:
\begin{enumerate}[label=$(\roman*)$,leftmargin=* ,parsep=2pt,itemsep=0pt,topsep=2pt]
	\item A \textit{moving control support} on $[0,T]$ is a family $(\omega(t))_{t\in[0,T]}$ of subsets of $\mathbb R^n$ such that the map $(t,x)\in[0,T]\times\mathbb R^n\mapsto\mathbbm 1_{\omega(t)}(x)$ is measurable.
	\item The control system \eqref{05072021E1} is said to be \textit{exactly null-controllable} in time $T$ from the moving control support $(\omega(t))_{t\in[0,T]}$ when for all $f_0\in L^2(\mathbb R^n)$, there exists a control $h\in L^2((0,T)\times\mathbb R^n)$ such that the mild solution of \eqref{05072021E1} satisfies $f(T,\cdot) = 0$.
	\item The control system \eqref{05072021E1} is said to be \textit{approximately null-controllable with uniform cost} in time $T$ from the moving control support $(\omega(t))_{t\in[0,T]}$ if for all $\varepsilon>0$, there exists a positive constant $C_{\varepsilon,T}>0$ such that for all $f_0\in L^2(\mathbb R^n)$, there exists a control $h\in L^2((0,T)\times\mathbb R^n)$ such that the mild solution of \eqref{05072021E1} satisfies 
	$$\Vert f(T,\cdot)\Vert_{L^2(\mathbb R^n)}\le\varepsilon\Vert f_0\Vert_{L^2(\mathbb R^n)},$$
	with moreover
	$$\int_0^T\Vert h(t,\cdot)\Vert^2_{L^2(\omega(t))}\, \mathrm dt\le C_{\varepsilon,T}\Vert f_0\Vert^2_{L^2(\mathbb R^n)}.$$
\end{enumerate}
We define similarly the notions of exact null-controllability and approximate null-controllability with uniform cost for the equation \eqref{25052018E3} in time $T>0$ from a fixed control support $\omega\subset\mathbb R^n$.

\subsection{Non-autonomous diffusive evolution equations} We are first interested in studying the cost-uniform approximate null-controllability of the evolution equation \eqref{05072021E1} in a diffusive setting. Precisely, fixing some positive time $T>0$, we assume that the following ellipticity condition holds for the family of time-dependent matrices $(Q_t)_{t\in\mathbb R}$: there exist a positive integer $k\geq1$ and a positive constant $c_T\in(0,1)$ such that for all $t\in[0,T]$ and $\xi\in\mathbb R^n$,
\begin{equation}\label{05072021E2}\tag{$A_T$}
	\int_t^TQ_s\xi\cdot\xi\, \mathrm ds\geq c_T(T-t)^k\vert\xi\vert^2.
\end{equation}
The main result contained in this work is the following

\begin{thm}\label{19112021T1} Let $(Q_t)_{t\in\mathbb R}$ be a family of real symmetric $n\times n$ matrices depending analytically on the time variable $t\in\mathbb R$, and $T>0$ be a positive time. Assume that the ellipticity condition \eqref{05072021E2} holds. Then, for all moving control support $(\omega(t))_{t\in[0,T]}$, the diffusive equation \eqref{05072021E1} is cost-uniformly approximately null-controllable in time $T$ from $(\omega(t))_{t\in[0,T]}$ if and only if there exist a radius $r>0$ and a rate $\gamma\in(0,1]$ such that
\begin{equation}\label{23112021E1}
	\forall x\in\mathbb R^n,\quad\frac1T\int_0^T\Leb\big(\omega(t)\cap B(x,r)\big)\, \mathrm dt\geq\gamma V_r.
\end{equation}
\end{thm}

\begin{rk}\label{07122021R1} The above geometric condition \eqref{23112021E1} has already been considered in the work \cite{BEPS} where a more general class of equations associated with non-autonomous Ornstein-Uhlenbeck operators is studied, which takes the form
\begin{equation}\label{07122021E2}\tag{$E_{B_t,Q_t}$}
\left\{\begin{aligned}
	& \partial_tf(t,x) + P_tf(t,x) = h(t,x)\mathbbm 1_{\omega(t)}(x),\quad (t,x)\in(0,T]\times\mathbb R^n, \\
	& f(0,\cdot) = f_0\in L^2(\mathbb R^n),
\end{aligned}\right.
\end{equation}
where the time-dependent operator $P_t$ is given by 
$$P_t = Q_tD_x\cdot D_x + B_tx\cdot\nabla_x,\quad x\in\mathbb R^n.$$
In the present work, we took the decision to consider only the equations \eqref{05072021E1}, since our main objective is to obtain a positive cost-uniform approximate null-controllability result for the hypoelliptic Ornstein-Uhlenbeck equations \eqref{25052018E3}, whose study can be deduced from the one of \eqref{05072021E1} for the particular matrices $Q_t$ defined by \eqref{06122021E2}. However, the strategy implemented in this paper can be easily adapted to deal with the equations \eqref{07122021E2}. By the way, notice that those matrices \eqref{06122021E2} turn out to be analytic with respect to the time variable $t\in\mathbb R$. This is the reason why we made the same assumption in Theorem \ref{19112021T1}. In fact, this regularity condition is crucial in the proof of this result, and we do not currently know how to relax it.
\end{rk}

\begin{rk} It follows from \cite[Theorem 1.5]{BEPS} that if the equation \eqref{07122021E2} is exactly null-controllable on $[0,T]$ from the moving control support $(\omega(t))_{t\in[0,T]}$, then there exist a radius $r>0$ and a rate $\gamma\in(0,1]$ such that
\begin{equation}\label{07122021E3}
	\forall x\in\mathbb R^n,\quad\frac1T\int_0^T\Leb\big(R(0,T-t)\omega(t)\cap B(x,r)\big)\, \mathrm dt\geq\gamma V_r,
\end{equation}
where $R$ stands for the resolvent of the following time-varying linear system
$$\dot X(t) = B_{T-t}X(t),\quad t\in[0,T].$$
The integral thickness condition \eqref{23112021E1} considered in this work corresponds to \eqref{07122021E3} when the matrices $B_t$ are all equal to zero. A natural question, asked in \cite{BEPS}, is then to wonder if the condition \eqref{23112021E1} is sufficient to derive exact null-controllability for the equation \eqref{05072021E1}. This is a very interested point that will not be tackled here, and therefore remains open. However, we provide a partial answer to the authors of \cite{BEPS} by proving that the integral thickness condition \eqref{23112021E1} is a necessary and sufficient geometric condition to ensure the cost-uniform approximate null-controllability of the equation \eqref{05072021E1}.
\end{rk}

\subsection{Hypoelliptic Ornstein-Uhlenbeck evolution equations} As an application of Theorem \ref{19112021T1}, we perform the study of the cost-uniform approximate null-controllability of the hypoelliptic Ornstein-Uhlenbeck equation \eqref{25052018E3}. We check in Section \ref{transition} that the study of the equation \eqref{25052018E3} reduces to the one of the equation \eqref{05072021E1} associated with the time-dependent matrices $Q_t$ given for all $t\in\mathbb R$ by  
\begin{equation}\label{06122021E2}
	Q_t = e^{(T-t)B}Qe^{(T-t)B^T},
\end{equation}
and from the moving control support $(\omega(t))_{t\in[0,T]}$ with $\omega(t) = e^{(T-t)B}\omega$. Moreover, we work in a hypoelliptic setting by assuming that the following Kalman rank condition holds
\begin{equation}\label{05072021E3}
	\Rank\big[B\ \vert\ \sqrt Q\big] = n,
\end{equation}
where 
$$\big[B\ \vert\ \sqrt Q\big] = \big[\sqrt Q,B\sqrt Q,\cdots,B^{n-1}\sqrt Q\big],$$
is the $n\times n^2$ matrix obtained by writing consecutively the columns of the matrices $B^j\sqrt Q$. As a consequence of the seminal work \cite{H}, this condition is known to be one of the characterizations of the hypoellipticity of the operator $P$, see e.g. the introduction of \cite{AB}. In Section \ref{transition}, we also check that under the Kalman rank condition \eqref{05072021E3}, an ellipticity condition of the form \eqref{05072021E2} holds for the matrices $Q_t$ given by \eqref{06122021E2}. As a consequence, we obtain a necessary and sufficient geometric condition on the support control $\omega\subset\mathbb R^n$ that ensures the cost-uniform approximate null-controllability of the hypoelliptic Ornstein-Uhlenbeck equation \eqref{25052018E3}, presented in the following statement.

\begin{cor}\label{06122021C1} Let $P$ be the Ornstein-Uhlenbeck operator defined in \eqref{12052021E1}. Assume that the Kalman rank condition \eqref{05072021E3} holds. Then, for all positive time $T>0$, the evolution equation \eqref{25052018E3} is cost-uniformly approximately null-controllable from the control support $\omega$ in time $T$ if and only if there exist a radius $r>0$ and a rate $\gamma\in(0,1]$ such that
\begin{equation}\label{06122021E1}
	\forall x\in\mathbb R^n,\quad\frac1T\int_0^T\Leb\big((e^{tB}\omega)\cap B(x,r)\big)\, \mathrm dt\geq\gamma V_r.
\end{equation}
\end{cor}

\begin{rk} Let us recall that a Borel set $\omega\subset\mathbb R^n$ is called \textit{thick} when there exist a radius $r>0$ and a rate $\gamma\in(0,1]$ such that 
\begin{equation}\label{07122021E4}
	\forall x\in\mathbb R^n,\quad\Leb\big(\omega\cap B(x,r)\big)\geq\gamma V_r.
\end{equation}
As explained in the introduction, the thickness condition is known to be a geometric necessary and sufficient condition to ensure the stabilization and the exact or approximate null-controllability with uniform cost of a large class of parabolic equations.  The above \textit{integral thickness condition} \eqref{06122021E1} generalizes the thickness property \eqref{07122021E4} and is well-adapted for the study of the null-controllability of hypoelliptic evolution equations, as illustrated in Corollary \ref{06122021C1}. Following the discussion started in Remark \ref{07122021R1}, one could legitimately wonder if the condition \eqref{06122021E1} allows to obtain positive exact null-controllability results for the equation \eqref{25052018E3}. This is still an interesting open question that will not be tackled in the present work. However, let us recall from \cite[Theorem 1.12]{AB} that when $\omega\subset\mathbb R^n$ is thick, there exists a positive constant $C>0$ such that for all $T>0$ and $g\in L^2(\mathbb R^n)$, the following exact observability estimate holds
$$\big\Vert e^{-TP_{co}}g\big\Vert^2_{L^2(\mathbb R^n)}\le C\exp\bigg(\frac C{T^{1+2k_0}}\bigg)\int_0^T\big\Vert e^{-tP_{co}}g\big\Vert^2_{L^2(\omega)}\, \mathrm dt,$$
where we set $P_{co} = P + \Tr(B)/2$, and where $0\le k_0\le n-1$ is the integer \eqref{06122021E3} intrinsically linked to the Kalman rank condition \eqref{05072021E3}. By the Hilbert Uniqueness Method, this implies that the equation \eqref{25052018E3} is exactly null-controllable from thick control supports in any positive time $T>0$. 
\end{rk}

\subsection{Examples} Let us now illustrate Theorem \ref{19112021T1} and Corollary \ref{06122021C1} by considering the same three examples as in the work \cite{BEPS}. 

The first two examples considered in this work are the Kolmogorov equation \eqref{12072021E1} and the Kolmogorov equation with an external quadratic force \eqref{12072021E2} in two dimensions, for which the flows generated by the matrix $B$ are respectively translations and rotations. Considering cones as control supports and using the fact that the integral thickness condition \eqref{06122021E1} is necessary to obtain positive exact null-controllability results for these two evolution equations, the authors of \cite{BEPS} exhibit a minimal time $T_0>0$ for which the equations \eqref{12072021E1} and  \eqref{12072021E2} are not exactly null-controllable in $[0,T]$ whenever $0<T\le T_0$. However, they can not conclude that these particular equations are exactly null-controllable on $[0,T]$ when $T>T_0$. In the present paper, we give a partial answer by proving that these equations are cost-uniformly approximately null-controllable on $[0,T]$ under the condition $T>T_0$.

\begin{ex}[Translation] We consider the Kolmogorov equation
\begin{equation}\label{12072021E1}
\left\{\begin{array}{l}
	(\partial_t - \partial^2_v+v\partial_x)f(t,x,v) = h(t,x,v)\mathbbm1_{\omega}(x,v),\ (t,x,v)\in(0,+\infty)\times\mathbb R^2,\\[5pt]
	f(0,\cdot) = f_0\in L^2(\mathbb R^2).
\end{array}\right.
\end{equation}
This is the equation \eqref{25052018E3} associated with the matrices
$$Q = \begin{pmatrix}
	0 & 0 \\
	0 & 1
\end{pmatrix}\quad\text{and}\quad B = \begin{pmatrix}
	0 & 1 \\
	0 & 0
\end{pmatrix}.$$
Notice that the Kalman rank condition \eqref{05072021E3} holds, and that the flow associated with the matrix $B$ is composed of translations given by
$$\forall t\geq0,\quad e^{tB} = \begin{pmatrix}
	1 & t \\
	0 & 1
\end{pmatrix}.$$
Let $0<\theta_0<\pi/2$ be an angle and $\omega_{\theta_0}\subset\mathbb R^2$ be the following cone
$$\omega_{\theta_0} = \big\{(x,\alpha x)\in\mathbb R^2 : -\tan\theta_0<\alpha<\tan\theta_0\big\}.$$
It follows from \cite[Proposition 2.3]{BEPS} that $\omega_{\theta_0}$ satisfies the integral thickness condition \eqref{06122021E1} associated with the above matrix $B$ if and only if $T>2/\tan\theta_0$, and so the equation \eqref{12072021E1} is not exactly null-controllable from $\omega_{\theta_0}$ when $T\le 2/\tan\theta_0$. However, we deduce from Corollary \ref{06122021C1} that for all positive time $T>0$, the Kolmogorov equation \eqref{12072021E1} is cost-uniformly approximately null-controllable from the control support $\omega_{\theta_0}$ in time $T$ if and only if $T > 2/\tan\theta_0$.
\end{ex}

\begin{ex}[Rotation] Let us now consider the Kolmogorov equation with external force
\begin{equation}\label{12072021E2}
\left\{\begin{array}{l}
	(\partial_t - \partial^2_v+v\partial_x-x\partial_v)f(t,x,v) = h(t,x,v)\mathbbm1_{\omega}(x,v),\ (t,x,v)\in(0,+\infty)\times\mathbb R^2,\\[5pt]
	f(0,\cdot) = f_0\in L^2(\mathbb R^2).
\end{array}\right.
\end{equation}
This is the equation \eqref{25052018E3} associated with the matrices
$$Q = \begin{pmatrix}
	0 & 0 \\
	0 & 1
\end{pmatrix}\quad\text{and}\quad B = \begin{pmatrix}
	0 & 1 \\
	-1 & 0
\end{pmatrix}.$$
Notice that the Kalman rank condition \eqref{05072021E3} holds, and that the flow associated with the matrix $B$ is composed of rotations given by
$$\forall t\geq0,\quad e^{tB} = \begin{pmatrix}
	\cos t & \sin t \\
	-\sin t & \cos t
\end{pmatrix}.$$
Let $0<\theta_0<\pi/4$ be an angle and $\omega_{\theta_0}\subset\mathbb R^2$ be the following cone
$$\omega_{\theta_0} = \big\{(x,\alpha x)\in\mathbb R^2 : 0<\alpha<\tan\theta_0\big\}.$$
As proven in \cite[Proposition 2.4]{BEPS}, the set $\omega_{\theta_0}$ satisfies the geometric condition \eqref{06122021E1} associated with the matrix $B$ when $T>\pi-\theta_0$, and does not in the case $T<\pi-\theta_0$. As a consequence, the equation \eqref{12072021E2} is not exactly null-controllable from $\omega_{\theta_0}$ when $T<\pi-\theta_0$. Let us check that $\omega_{\theta_0}$ also fails to satisfy this geometric condition when $T=\pi-\theta_0$. First of all, notice that for all $0 \leq t \leq \pi-\theta_0$, the set $e^{t B}\omega_{\theta_0}$ is the cone $\omega_{\theta_0}$ rotated with angle $-t$. It follows that for all $0< t < \frac{\pi}2$, 
\begin{equation}\label{cone1}
	e^{tB}\omega_{\theta_0} \cap \big\{(x,\alpha x)\in\mathbb R^2 : \tan (\theta_0-t)<\alpha \big\}=\emptyset,
\end{equation}
and for all $\frac{\pi}{2} < t < \pi- \theta_0$, 
\begin{equation}\label{cone2}
	e^{tB}\omega_{\theta_0} \cap \big\{(x,\alpha x)\in\mathbb R^2 : 0<\alpha<\tan (\pi-t)\big\}=\emptyset.
\end{equation}
Let $0<t<\pi-\theta_0$ and $r>0$. We deduce from \eqref{cone1} and \eqref{cone2} that for $n\gg1$ sufficiently large, we have
$$(e^{tB}\omega_{\theta_0})\cap B\big((n,n \tan \theta_0),r\big) = \emptyset.$$
This implies, by Lebesgue's dominated convergence theorem, that
$$\lim_{n \to +\infty} \int_0^{\pi-\theta_0} \Leb\big((e^{tB}\omega_{\theta_0}) \cap B((n,n \tan \theta_0), r)\big)\,\mathrm dt =0.$$
As we claimed, the set $\omega_{\theta_0}$ therefore fails to satisfy the geometric condition \eqref{06122021E1} associated with the matrix $B$ when $T=\pi-\theta_0$. Corollary \ref{06122021C1} then implies that for all positive time $T>0$, the Kolmogorov equation with quadratic external force \eqref{12072021E2} is cost-uniformly approximately null-controllable from the control support $\omega_{\theta_0}$ in time $T$ if and only if $T>\pi-\theta_0$.
\end{ex}

The last example considered in \cite{BEPS} deals with the heat equation in one dimension with a dilating moving control support which has the particularity to be constituted of non-thick subsets of $\mathbb R^n$, but which satisfies the integral thickness condition \eqref{23112021E1} for any positive time $T>0$. As before, the question stated by the authors of \cite{BEPS} is to know whether this equation is exactly null-controllable or not. As in the two first examples, we give a partial answer by checking that this equation is cost-uniformly approximately null-controllable in any positive time $T>0$ (there is no minimal time there).

\begin{ex}[Dilatation] In this last example, we consider the heat equation in dimension 1 with a moving control support
\begin{equation}\label{12072021E3}
\left\{\begin{array}{l}
	(\partial_t - \partial^2_x)f(t,x) = h(t,x)\mathbbm1_{\omega(t)}(x),\ (t,x)\in(0,+\infty)\times\mathbb R,\\[5pt]
	f(0,\cdot) = f_0\in L^2(\mathbb R),
\end{array}\right.
\end{equation}
where, setting $\mu>0$ a positive real number, the Borel subsets $\omega(t)\subset\mathbb R$ are defined for all $t\geq0$ by
$$\omega(t) = \sqrt{1+2\mu t}\,\omega,\quad\omega = [-1,1]\cup\bigcup_{n\geq1}(n^2,n^2+n)\cup(-n^2-n,-n^2).$$
This is of course the equation \eqref{05072021E1} when $Q_t = 1$ for all $t\geq0$, and the ellipticity condition \eqref{05072021E2} is satisfied for all $T>0$. It is proven in \cite[Subsection 2.5]{BEPS} that for all $t\geq0$, the set $\omega(t)$ is not thick, and also in Proposition 2.5 of the same work that for all $T>0$, the moving control support $(\omega(t))_{t\in[0,T]}$ satisfies the integral thickness condition \eqref{23112021E1}. Theorem \ref{19112021T1} therefore implies that the heat equation \eqref{12072021E3} is cost-uniformly approximately null-controllable from the moving control support $(\omega(t))_{t\in[0,T]}$ in any positive time $T>0$. Notice that as a consequence of Proposition \ref{07122021P1}, stated in the next paragraph, the same result holds for the equation \eqref{12072021E3} where $-\partial^2_x$ is replaced by the fractional Laplacian $(-\partial^2_x)^s$, with $s\geq1/2$ a positive real number.
\end{ex}

\subsection{Heuristics} Let us now present the strategy of the proof of Theorem \ref{19112021T1}. The key step consists in using the fact that the cost-uniform approximate null-controllability of the equation \eqref{05072021E1} in time $T>0$ is equivalent to the following weak observability estimate, see Corollary \ref{12012022C1} in the appendix
\begin{multline}\label{06122021E4}
	\forall\varepsilon\in(0,1), \exists C_{\varepsilon,T}>0, \forall g\in L^2(\mathbb R^n), \\
	\big\Vert U(T,0)g\big\Vert^2_{L^2(\mathbb R^n)}\le C_{\varepsilon,T}\int_0^T\big\Vert U(T,t)g\big\Vert^2_{L^2(\omega(t))}\, \mathrm dt + \varepsilon\Vert g\Vert^2_{L^2(\mathbb R^n)},
\end{multline}
where the Fourier multiplier $U(T,t)$ is given by
\begin{equation}\label{24112021E1}
	U(T,t) = \exp\bigg(-\int_t^TQ_sD_x\cdot D_x\, \mathrm ds\bigg).
\end{equation}
On the one hand, by propagating a Gaussian function in this observability estimate, we check in Section \ref{necessary} that the geometric condition \eqref{23112021E1} is necessary to obtain a positive cost-uniform approximate null-controllability result for the equation \eqref{05072021E1}.

In order to prove that this geometric condition is also sufficient, we adapt the strategy used in \cite[Subsection 5.2]{AM} where the authors proved that the thickness property \eqref{07122021E4} is a necessary and sufficient condition that ensures the cost-uniform approximate null-controllability for a large class of parabolic equations. In the present work, we begin by noticing that the geometric condition \eqref{23112021E1} implies that the set $\Omega\subset[0,T_{\gamma}]\times\mathbb R^n$ defined as follows
$$\Omega= \Big\{(t,x)\in[0,T_{\gamma}]\times\mathbb R^n : x \in \omega(t)\Big\},$$
is a thick subset of $[0,T_{\gamma}]\times\mathbb R^n$, with $T_{\gamma} = (1-\gamma/2)T$ and where the rate $\gamma\in(0,1]$ is the one appearing in \eqref{23112021E1}. Notice that we have to get strictly far from the final time $T$ in order to avoid blow-up phenomena. This is the precise reason why we can use the same strategy as in the work \cite{AM}. First, we need to establish smoothing estimates in the time and space variables of the following form, by using the ellipticity condition \eqref{05072021E2},
\begin{equation}\label{06122021E6}
	\big\Vert\partial^m_t\partial^{\alpha}_x(U(T,t)g)\big\Vert_{L^2(\mathbb R^n)}\le c_0^{m+\vert\alpha\vert}\ \bigg(\frac{C_T}{T-t}\bigg)^{\frac k2(2m+\vert\alpha\vert)}\ m!\ \sqrt{\alpha!}\ \Vert g\Vert_{L^2(\mathbb R^n)},
\end{equation}
with $C_T = \max(1,T)s_T/c_T$, $c_T\in(0,1)$ and $k\geq1$ being the ones involved in \eqref{05072021E2}, and $s_T>1$ being a positive constant related to the analyticity property of the family $(Q_t)_{t\in\mathbb R}$ on $(-T,T)$. These estimates are obtained in Section \ref{smooth}. Notice that when we work with the thickness condition \eqref{07122021E4}, which does not depend on time, as in the work \cite{AM}, we only have to consider the smoothing properties in space of the evolution equation at play. Then, by using the above estimates and elements of harmonic analysis, and more precisely the unique continuation property stated in Proposition \ref{30112021P1}, coming essentially from the second author's work \cite{M}, we obtain the following quantitative unique continuation property in Section \ref{suffcond}:
$$\int_0^{T_{\gamma}}\big\Vert U(T,t)g\big\Vert^2_{L^2(\mathbb R^n)}\ \mathrm dt \le \bigg(\frac{K_n(2-\gamma)}{\gamma}\bigg)^{K_nC}\int_0^{T_{\gamma}} \big\Vert U(T,t)g\big\Vert^2_{L^2(\omega(t))}\, \mathrm dt+ \varepsilon \Vert g\Vert^2_{L^2(\mathbb R^n)},$$
where the positive constant $K_n>0$ only depends on the dimension $n$, and where the other positive constant $C = C_{\varepsilon,\gamma,r,k,T}>0$ is given by
$$C = \bigg(1-\log(\varepsilon r^n) +\log \bigg(1+\frac{C^{2k}_T}{\gamma^{2k}T^{2(k-1)}}+ \frac{r^2C^k_T}{\gamma^kT^k} \bigg)\bigg)\exp\bigg(\frac{K_nC^k_T}{\gamma^kT^{k-1}}\bigg)+ \frac{r^2C^k_T}{\gamma^kT^k},$$
where $r>0$ is the radius appearing in \eqref{23112021E1}. The weak observability estimate \eqref{06122021E4} is then deduced from a monotonicity argument.

Notice that our strategy can be adapted to deal with other equations than the one studied in the present work, and in particular fractional diffusive models. Indeed, let us consider $s>0$ a positive real number and the following associated fractional heat equation posed on the whole Euclidean space
\begin{equation}\label{07122021E1}\tag{$E_s$}
\left\{\begin{aligned}
	& \partial_tf(t,x) + (-\Delta)^sf(t,x) = h(t,x)\mathbbm 1_{\omega(t)}(x),\quad (t,x)\in(0,T]\times\mathbb R^n, \\
	& f(0,\cdot) = f_0\in L^2(\mathbb R^n),
\end{aligned}\right.
\end{equation}
where $T>0$ is a final time and $(\omega(t))_{t\in[0,T]}$ is a moving control support. By passing to the Fourier side and using the same estimates as in Section \ref{smooth}, one can easily check that there exists a positive constant $c>0$ such that for all $m\geq0$, $\alpha\in\mathbb N^n$, $t>0$ and $g\in L^2(\mathbb R^n)$,
$$\big\Vert\partial^m_t\partial^{\alpha}_x(e^{-t(-\Delta)^s}g)\big\Vert_{L^2(\mathbb R^n)}\le \frac{c^{m+\vert\alpha\vert}}{t^{m+\frac{\vert\alpha\vert}{2s}}}\ m!\ (\alpha!)^{\frac1{2s}}\ \Vert g\Vert_{L^2(\mathbb R^n)}.$$
As a consequence, by assuming that $s\geq1/2$ and using Proposition \ref{30112021P1} when $s>1/2$, Proposition \ref{17092021P1} when $s=1/2$, and the same steps as in Section \ref{suffcond}, one can obtain an uncertainty principle of the form \eqref{06122021E4} for the semigroup generated by the fractional Laplacian $(-\Delta)^s$, and therefore conclude to the following proposition (the necessary part follows exactly the same steps as in Section \ref{necessary})

\begin{prop}\label{07122021P1} Let $s\geq1/2$ be a positive real number. For all positive time $T>0$ and all moving control support $(\omega(t))_{t\in[0,T]}$, the fractional heat equation \eqref{07122021E1} is cost-uniformly approximately null-controllable in time $T$ from $(\omega(t))_{t\in[0,T]}$ if and only if the integral thickness condition \eqref{23112021E1} holds.
\end{prop}

More generally, the authors conjecture that Theorem \ref{19112021T1} could be extended to the more general class of fractional diffusive equations, where $s\geq1/2$ is a positive real number,
$$\left\{\begin{aligned}
	& \partial_tf(t,x) + (Q_tD_x\cdot D_x)^sf(t,x) = h(t,x)\mathbbm 1_{\omega(t)}(x),\quad t\in[0,T],\ x\in\mathbb R^n, \\
	& f(0,\cdot) = f_0\in L^2(\mathbb R^n).
\end{aligned}\right.$$
This would extend Corollary \ref{06122021C1} for the class of hypoelliptic fractional Ornstein-Uhlenbeck operators
$$(QD_x\cdot D_x)^s + Bx\cdot\nabla_x,\quad x\in\mathbb R^n.$$
However, the strategy implemented in the present paper does not allow to treat this general class of equations, since we can not obtain the smoothing estimates \eqref{06122021E6} in the fractional setting (except for the fractional heat equations as explained, or when $s\in\mathbb N^*$ is a positive integer), due to the fact that \textit{a priori}, there is a lack of regularity with respect to the time variable $t\in[0,T)$ for the associated fractional Fourier multipliers
$$U_s(T,t) = \exp\bigg(-\int_t^T(Q_\tau D_x\cdot D_x)^s\, \mathrm d\tau\bigg).$$

\section{From the hypoelliptic Ornstein-Uhlenbeck equations to non-autonomous diffusive equations}\label{transition}

Let $P$ be the Ornstein-Uhlenbeck operator defined in \eqref{12052021E1}. This section is devoted to check that that the study of the cost-uniform approximate null-controllability of the equation \eqref{25052018E3} can be deduced from the study of a specific time-dependent diffusive equation \eqref{05072021E1}. The strategy is to use the interpretation of the cost-uniform approximate null-controllability in terms of weak observability estimate. In the following, we consider the maximal realization of the operator $P$ on $L^2(\mathbb R^n)$, that is, the operator $P$ is equipped with the domain
$$D(P) = \big\{u\in L^2(\mathbb R^n) : Pu\in L^2(\mathbb R^n)\big\}.$$

First of all, we deduce from the change of unknown $\tilde f(t,x) = e^{\Tr(B)t/2}f(t,x)$ that the cost-uniformly approximate null-controllability of the equation \eqref{25052018E3} is equivalent to the one of the equation
\begin{equation}\label{23072021E1}\tag{$E_{P_{co}}$}
\left\{\begin{array}{l}
	\partial_tf(t,x)+P_{co}f(t,x) = h(t,x)\mathbbm1_{\omega}(x),\quad (t,x)\in(0,+\infty)\times\mathbb R^n,\\[5pt]
	f(0,\cdot) = f_0\in L^2(\mathbb R^n),
\end{array}\right.
\end{equation}
where the operator $P_{co}$ is given by
$$P_{co} = QD_x\cdot D_x + Bx\cdot\nabla_x + \frac12\Tr(B),\quad x\in\mathbb R^n.$$
It follows from Proposition 6 in \cite{MR4041279} that the equation \eqref{23072021E1} is cost-uniformly approximately null-controllable from the control support $\omega\subset\mathbb R^n$ in time $T>0$ if and only if for all $\varepsilon\in(0,1)$, there exists a positive constant $C_{\varepsilon,T}>0$ such that for all $g\in L^2(\mathbb R^n)$,
\begin{equation}\label{06112020E6}
	\big\Vert e^{-TP_{co}^*}g\big\Vert^2_{L^2(\mathbb R^n)}\le C_{\varepsilon, T}\int_0^T\big\Vert e^{-tP_{co}^*}g\big\Vert^2_{L^2(\omega)}\, \mathrm dt 
	+ \varepsilon\Vert g\Vert^2_{L^2(\mathbb R^n)},
\end{equation}
where $P_{co}^*$ stands for the adjoint of the operator $P_{co}$ in $L^2(\mathbb R^n)$. The first author proved in collaboration with J. Bernier in \cite[Corollary 2.2]{AB} that this adjoint is given by
$$P_{co}^* = QD_x\cdot D_x - Bx\cdot\nabla_x - \frac12\Tr(B),\quad x\in\mathbb R^n,$$
with domain
$$D(P_{co}^*) = \big\{u\in L^2(\mathbb R^n) : P_{co}^*u\in L^2(\mathbb R^n)\big\}.$$
Moreover, in the same work \cite[Theorem 1.1]{AB}, we proved that the evolution operators $e^{-tP_{co}^*}$ are given by the following explicit formulas for all $t\geq0$,
\begin{equation}\label{26102022E1}
	e^{-tP_{co}^*} = e^{\frac12\Tr(B)t}\exp\bigg(-\int_0^t\big\vert\sqrt Qe^{-sB^T}D_x\big\vert^2\, \mathrm ds\bigg)\, e^{tBx\cdot\nabla_x}.
\end{equation}
By using that for all $u\in L^2(\mathbb R^n)$,
\begin{equation}\label{01122021E4}
	e^{tBx\cdot\nabla_x}u = u(e^{tB}\,\cdot),\quad\text{and therefore,}\quad\mathscr F(e^{tBx\cdot\nabla_x}u) = e^{-\Tr(B)t}\mathscr F(u)(e^{-tB^T}\,\cdot),
\end{equation}
where $\mathscr F$ denotes the Fourier transform, and a change of variable in time in the integral, we deduce that for all $\xi\in\mathbb R^n$,
\begin{align*}
	\mathscr F(e^{-tP_{co}^*}u)(\xi) & = e^{\frac12\Tr(B)t}\exp\bigg(-\int_0^t\big\vert\sqrt Qe^{-sB^T}\xi\big\vert^2\, \mathrm ds\bigg)\,e^{-\Tr(B)t}\mathscr F(u)(e^{-tB^T}\xi) \\[5pt]
	& = e^{\frac12\Tr(B)t}\exp\bigg(-\int_0^t\big\vert\sqrt Qe^{sB^T}e^{-tB^T}\xi\big\vert^2\, \mathrm ds\bigg)\,e^{-\Tr(B)t}\mathscr F(u)(e^{-tB^T}\xi) \\[5pt]
	& = e^{\frac12\Tr(B)t}\,e^{-\Tr(B)t}\mathscr F\bigg(\exp\bigg(-\int_0^t\big\vert\sqrt Qe^{sB^T}D_x\big\vert^2\, \mathrm ds\bigg)u\bigg)(e^{-tB^T}\xi) \\[5pt]
	& = e^{\frac12\Tr(B)t}\,\mathscr F\bigg(e^{tBx\cdot\nabla_x}\exp\bigg(-\int_0^t\big\vert\sqrt Qe^{sB^T}D_x\big\vert^2\, \mathrm ds\bigg)u\bigg)(\xi).
\end{align*}
As a consequence, formula \eqref{26102022E1} can be rewritten in the following way
\begin{equation}\label{01122021E2}
	e^{-tP_{co}^*} = e^{\frac12\Tr(B)t}e^{tBx\cdot\nabla_x}\,\exp\bigg(-\int_0^t\big\vert\sqrt Qe^{sB^T}D_x\big\vert^2\, \mathrm ds\bigg).
\end{equation}
It is this representation that will be useful for us in the present work. Let us plug the formula \eqref{01122021E2} into the weak observability estimate \eqref{06112020E6}. On the one hand, by using the first equality in \eqref{01122021E4}, we deduce from successive changes of variables that for all $g\in L^2(\mathbb R^n)$,
\begin{align*}
	\int_0^T\big\Vert e^{-tP_{co}^*}g\big\Vert^2_{L^2(\omega)}\,\mathrm dt 
	& = \int_0^T\bigg\Vert e^{\frac12\Tr(B)t}e^{tBx\cdot\nabla_x}\,\exp\bigg(-\int_0^t\big\vert\sqrt Qe^{sB^T}D_x\big\vert^2\, \mathrm ds\bigg)g\bigg\Vert^2_{L^2(\omega)}\,\mathrm dt \\[5pt]
	& = \int_0^T\bigg\Vert\exp\bigg(-\int_0^t\big\vert\sqrt Qe^{sB^T}D_x\big\vert^2\, \mathrm ds\bigg)g\bigg\Vert^2_{L^2(e^{tB}\omega)}\,\mathrm dt \\[5pt]
	& = \int_0^T\bigg\Vert\exp\bigg(-\int_0^{T-t}\big\vert\sqrt Qe^{sB^T}D_x\big\vert^2\, \mathrm ds\bigg)g\bigg\Vert^2_{L^2(e^{(T-t)B}\omega)}\,\mathrm dt \\[5pt]
	& = \int_0^T\bigg\Vert\exp\bigg(-\int_t^T\big\vert\sqrt Qe^{(T-s)B^T}D_x\big\vert^2\, \mathrm ds\bigg)g\bigg\Vert^2_{L^2(e^{(T-t)B}\omega)}\,\mathrm dt. \\[5pt]
	& = \int_0^T\big\Vert U(T,t)g\big\Vert^2_{L^2(\omega(t))}\,\mathrm dt,
\end{align*}
where we set $\omega(t) = e^{(T-t)B}\omega$, and where the Fourier multipliers $U(T,t)$ are the ones defined in \eqref{24112021E1} and associated with the matrices $Q_t$ given for all $t\in[0,T]$ by 
\begin{equation}\label{04012022E1}
	Q_t = e^{(T-t)B}Qe^{(T-t)B^T}.
\end{equation}
On the other hand, since the operators $e^{\Tr(B)t/2}e^{tBx\cdot\nabla_x}$ are unitary on $L^2(\mathbb R^n)$, we get that for all $g\in L^2(\mathbb R^n)$,
$$\big\Vert e^{-TP_{co}^*}g\big\Vert^2_{L^2(\mathbb R^n)} = \bigg\Vert\exp\bigg(-\int_0^T\big\vert\sqrt Qe^{sB^T}D_x\big\vert^2\, \mathrm ds\bigg) g\bigg\Vert^2_{L^2(\mathbb R^n)} = \big\Vert U(T,0)g\big\Vert^2_{L^2(\mathbb R^n)}.$$
As a consequence, the weak observability estimate \eqref{06112020E6} can be rewritten in the following way
$$\big\Vert U(T,0)g\big\Vert^2_{L^2(\mathbb R^n)}\le C_{\varepsilon, T}\int_0^T\big\Vert U(T,t)g\big\Vert^2_{L^2(\omega(t))}\, \mathrm dt 
+ \varepsilon\Vert g\Vert^2_{L^2(\mathbb R^n)},$$
and Corollary \ref{12012022C1} implies that the cost-uniform approximate null-controllability of the Ornstein-Uhlenbeck equation \eqref{25052018E3} is equivalent to the cost-uniform approximate null-controllability of the following non-autonomous equation
$$\left\{\begin{aligned}
	& \partial_tf(t,x) + Q_tD_x\cdot D_xf(t,x) = h(t,x)\mathbbm 1_{\omega(t)}(x),\quad t>0,\ x\in\mathbb R^n, \\
	& f(0,\cdot) = f_0\in L^2(\mathbb R^n).
\end{aligned}\right.$$

It now only remains to check that the Kalman rank condition \eqref{05072021E3} implies the ellipticity condition \eqref{05072021E2} for the above matrices $Q_t$. To that end, let us introduce the vector space $S\subset\mathbb R^n$ defined by
$$S = \bigcap_{j=0}^{+\infty}\Ker\big(\sqrt Q(B^T)^j\big).$$
On the one hand, it can be checked, see e.g. \cite[Lemma 6.1]{AB}, that the Kalman rank condition \eqref{05072021E3} is equivalent to the fact that $S$ is reduced to $\{0\}$. Moreover, it follows from \cite[Proposition 4.2]{A} that there exist some positive constants $c_0\in(0,1)$ and $T_0>0$ such that for all $T\in[0,T_0]$ and $\xi\in\mathbb R^n$,
$$\int_0^T\big\vert\sqrt Qe^{sB^T}\xi\big\vert^2\,\mathrm ds\geq c_0\sum_{k=0}^{k_0}T^{2j+1}\big\vert\sqrt Q(B^T)^j\xi\big\vert^2,$$
where the integer $0\le k_0\le n-1$ is defined by
\begin{equation}\label{06122021E3}
	k_0 = \min\bigg\{k\geq0 : S = \bigcap_{j = 0}^k\Ker(\sqrt Q(B^T)^j)\bigg\}.
\end{equation}
Notice that the fact that $0\le k_0\le n-1$ is a consequence of Cayley-Hamilton's theorem. Therefore, when the Kalman rank condition \eqref{05072021E3} holds, or equivalently, when the vector space $S$ is reduced to $\{0\}$, there exists another positive constant $c_1\in(0,1)$ such that for all $T\in[0,T_0]$ and $\xi\in\mathbb R^n$,
\begin{equation}\label{04012022E2}
	\int_0^T\big\vert\sqrt Qe^{sB^T}\xi\big\vert^2\,\mathrm ds\geq c_1T^{2k_0+1}\vert\xi\vert^2.
\end{equation}
Let us now consider $T>0$ and $t\in[0,T]$. When $0\le T-t\le T_0$, then the estimate \eqref{04012022E2} holds at time $T-t$. In the situation where $T-t>T_0$, we deduce from \eqref{04012022E2} that
$$\int_0^{T-t}\big\vert\sqrt Qe^{sB^T}\xi\big\vert^2\,\mathrm ds\geq\int_0^{T_0}\big\vert\sqrt Qe^{sB^T}\xi\big\vert^2\,\mathrm ds\geq c_1T_0^{2k_0+1}\vert\xi\vert^2 \geq c_1\bigg(\frac{T_0}T\bigg)^{2k_0+1}(T-t)^{2k_0+1}\vert\xi\vert^2.$$
Setting $c_T = c_1\min(1,(T_0/T)^{2k_0+1})\in(0,1)$, we therefore obtain that for all $T>0$ and $t\in[0,T]$,
$$\int_0^{T-t}\big\vert\sqrt Qe^{sB^T}\xi\big\vert^2\,\mathrm ds\geq c_T(T-t)^{2k_0+1}\vert\xi\vert^2.$$
Performing the change of variable $s' = T-s$ in the integral and using the definition \eqref{04012022E1} of the matrices $Q_t$, we therefore deduce that the ellipticity condition \eqref{05072021E2} holds for the family $(Q_t)_{t\in\mathbb R}$ with $k = 2k_0+1\geq1$ (which is uniform with respect to $T>0$).

\section{Necessary condition for approximate null-controllability from moving control supports}\label{necessary}

This section is devoted to the proof of the reciprocal part of Theorem \ref{19112021T1}, which provides a necessary geometric condition on the moving control support $(\omega(t))_{t\in[0,T]}$ so that the time-dependent diffusive equation \eqref{05072021E1} is cost-uniformly approximately null-controllable in time $T>0$ from $(\omega(t))_{t\in[0,T]}$. Notice that we do not make any assumption of regularity in time for the family $(Q_t)_{t\in\mathbb R}$, and we do not assume that the ellipticity assumption \eqref{05072021E2} holds. Precisely, we will establish the following

\begin{thm}\label{19112021T2} Let $T>0$ and $(\omega(t))_{t\in[0,T]}$ be a moving control support. If the evolution equation \eqref{05072021E1} is cost-uniformly approximately null-controllable in time $T$ from $(\omega(t))_{t\in[0,T]}$, then there exist a radius $r>0$ and a rate $\gamma\in(0,1]$ so that
$$\forall x\in\mathbb R^n,\quad\frac1T\int_0^T\Leb\big(\omega(t)\cap B(x,r)\big)\,\mathrm dt\geq\gamma V_r.$$
\end{thm}

\begin{proof} According to Corollary \ref{12012022C1}, assuming that the equation \eqref{05072021E1} is cost-uniformly approximately null-controllable in time $T$ from $(\omega(t))_{t\in[0,T]}$ is equivalent to assuming that for all $\varepsilon\in(0,1)$, there exists a positive constant $C_{\varepsilon,T}>0$ such that for all $g\in L^2(\mathbb R^n)$,
\begin{equation}\label{19122019E4}
	\big\Vert U(T,0)g\big\Vert^2_{L^2(\mathbb R^n)}\le C_{\varepsilon,T}\int_0^T\big\Vert U(T,t)g\big\Vert^2_{L^2(\omega(t))}\,\mathrm dt 
	+ \varepsilon\Vert g\Vert^2_{L^2(\mathbb R^n)},
\end{equation}
where $U(T,t)$ denotes the Fourier multiplier \eqref{24112021E1}. The strategy, which is very classical, consists in applying this observability estimate for well-chosen functions $g\in L^2(\mathbb R^n)$. In the following, we fix $\varepsilon\in(0,1)$ and we consider the associated positive constant $C_{\varepsilon,T}>0$. Also fixing $x_0\in\mathbb R^n$ and considering $l\gg1$ whose value will be adjusted later, we define the Gaussian function $g_l$ defined by
$$\forall x\in\mathbb R^n,\quad g_l(x) = \frac1{l^n}\exp\bigg(\frac{\vert x-x_0\vert^2}{2l^2}\bigg).$$
Classical results concerning Fourier transform of Gaussian functions show that
\begin{equation}\label{06112020E2}
	\forall\xi\in\mathbb R^n,\quad\widehat g_l(\xi) = (2\pi)^{n/2}\exp\bigg(-ix_0\cdot\xi-\frac{l^2\vert\xi\vert^2}2\bigg).
\end{equation}
In the following, we will use the notation
$$K(t,t_0,\xi) = \exp\bigg(-\int_{t_0}^tQ_s\xi\cdot\xi\,\mathrm ds\bigg).$$
On the one hand, it follows from Plancherel's theorem that the left-hand side of the inequality \eqref{19122019E4} applied to the functions $g_l$ is a positive constant independent on the point $x_0$, denoted $\delta_l>0$ in the following and given by
\begin{align}\label{19122019E3}
	\delta_l = \big\Vert U(T,0)g_l\big\Vert^2_{L^2(\mathbb R^n)} 
	& = \int_{\mathbb R^n}\big\vert e^{-ix_0\cdot\xi}K(T,0,\xi)e^{-l^2\vert\xi\vert^2/2}\big\vert^2\,\mathrm d\xi \\
	& = \frac1{l^n}\int_{\mathbb R^n}\big\vert K(T,0,\xi/l)e^{-\vert\xi\vert^2/2}\big\vert^2\,\mathrm d\xi>0. \nonumber
\end{align}
On the other hand, we get that the $L^2$-norm of the function $g_l$ also does not depend on the point $x_0\in\mathbb R^n$ and is given by the following Gaussian integral
\begin{equation}\label{27012021E1}
	\Vert g_l\Vert^2_{L^2(\mathbb R^n)} = \frac1{l^{2n}}\int_{\mathbb R^n}e^{-\vert x\vert^2/l^2}\,\mathrm dx = \bigg(\frac{\pi}{l^2}\bigg)^{n/2}.
\end{equation}
Let us check that the large positive parameter $l\gg1$ can be adjusted so that $\delta_l - \varepsilon\Vert g_l\Vert^2_{L^2(\mathbb R^n)}>0$, that is, by \eqref{19122019E3} and \eqref{27012021E1},
\begin{equation}\label{06102020E1}
	\int_{\mathbb R^n}\big\vert K(T,0,\xi/l)e^{-\vert\xi\vert^2/2}\big\vert^2\,\mathrm d\xi>\varepsilon\pi^{n/2}.
\end{equation}
Since the function $K(T,0,\cdot)$ is bounded and continuous, the dominated convergence theorem together with the fact that $\varepsilon\in(0,1)$ and $K(T,0,0) = 1$ implies that
\begin{align*}
	\lim_{l\rightarrow+\infty}\int_{\mathbb R^n}\big\vert K(T,0,\xi/l)e^{-\vert\xi\vert^2/2}\big\vert^2\,\mathrm d\xi
	& = K(T,0,0)^2\int_{\mathbb R^n}\big\vert e^{-\vert\xi\vert^2/2}\big\vert^2\,\mathrm d\xi \\[5pt]
	& = K(T,0,0)^2\pi^{n/2} >\varepsilon\pi^{n/2}. 
\end{align*}
The parameter $l\gg1$ can therefore be adjusted so that \eqref{06102020E1} holds. The value of $l\gg1$ is now fixed. We therefore deduce from \eqref{19122019E4} and \eqref{06102020E1} that
\begin{equation}\label{16082020E1}
	M_l\le C_{\varepsilon,T}\int_0^T\big\Vert U(T,t)g_l\big\Vert^2_{L^2(\omega(t))}\,\mathrm dt,
\end{equation}
with
$$M_l = \delta_l - \varepsilon\Vert g_l\Vert^2_{L^2(\mathbb R^n)}>0.$$
Moreover, by introducing $\mathscr F_{\xi}^{-1}$ the partial inverse Fourier transform with respect to the variable $\xi\in\mathbb R^n$ and using \eqref{06112020E2}, the right-hand side of this inequality (up to the constant $C_{\varepsilon,T}$) writes
\begin{align*}
	\int_0^T\big\Vert U(T,t)g_l\big\Vert^2_{L^2(\omega(t))}\ \mathrm dt 
	& = (2\pi)^n \int_0^T\int_{\omega(t)}\big\vert\mathscr F^{-1}_{\xi}(e^{-ix_0\cdot\xi}K(T,t,\xi)e^{-l^2\vert\xi\vert^2/2})(x)\big\vert^2\,\mathrm dx\,\mathrm dt \\[5pt]
	& = (2\pi)^n \int_0^T\int_{\omega(t)}\big\vert\mathscr F^{-1}_{\xi}(K(T,t,\xi)e^{-l^2\vert\xi\vert^2/2})(x-x_0)\big\vert^2\,\mathrm dx\,\mathrm dt \\[5pt]
	& = (2\pi)^n \int_0^T\int_{\omega(t)-x_0}\big\vert\mathscr F^{-1}_{\xi}(K(T,t,\xi)e^{-l^2\vert\xi\vert^2/2})(x)\big\vert^2\,\mathrm dx\,\mathrm dt.
\end{align*}
Given $r>0$ a positive radius whose value will be chosen later, we split the previous integral in two parts and obtain the following estimate:
\begin{multline}\label{23082018E6}
	\int_0^T\big\Vert U(T,t)g_l\big\Vert^2_{L^2(\omega(t))}\,\mathrm dt \\[5pt]
	\le (2\pi)^n\int_0^T\int_{(\omega(t)-x_0)\cap B(0,r)}\big\vert\mathscr F^{-1}_{\xi}(K(T,t,\xi)e^{-l^2\vert\xi\vert^2/2})(x)\big\vert^2\,\mathrm dx\,\mathrm dt \\[5pt]
	+ (2\pi)^n\int_0^T\int_{\vert x\vert>r}\big\vert\mathscr F^{-1}_{\xi}(K(T,t,\xi)e^{-l^2\vert\xi\vert^2/2})(x)\big\vert^2\,\mathrm dx\,\mathrm dt.
\end{multline}
Now, we study one by one the two integrals appearing in the right-hand side of \eqref{23082018E6}. First, notice that for all $0\le t\le T$,
\begin{align*}
	\big\Vert\mathscr F^{-1}_{\xi}(K(T,t,\xi)e^{-l^2\vert\xi\vert^2/2})\big\Vert_{L^{\infty}(\mathbb R^n)}
	& \le\frac1{(2\pi)^n}\big\Vert K(T,t,\xi)e^{-l^2\vert\xi\vert^2/2}\big\Vert_{L^1(\mathbb R^n)} \\[5pt]
	& \le\frac1{(2\pi)^n}\big\Vert e^{-l^2\vert\xi\vert^2/2}\big\Vert_{L^1(\mathbb R^n)} \\[5pt]
	& = \frac1{(2\pi)^n}\bigg(\frac{2\pi}{l^2}\bigg)^{n/2}.
\end{align*}
It therefore follows from the invariance by translation of the Lebesgue measure that
\begin{multline}\label{23082018E7}
	(2\pi)^n\int_0^T\int_{(\omega(t)-x_0)\cap B(0,r)}\big\vert\mathscr F^{-1}_{\xi}(K(T,t,\xi)e^{-l^2\vert\xi\vert^2/2})(x)\big\vert^2\,\mathrm dx\,\mathrm dt  \\[5pt]
	\le \frac1{l^{2n}}\int_0^T\Leb\big((\omega(t)-x_0)\cap B(0,r)\big)\,\mathrm dt = \frac1{l^{2n}}\int_0^T\Leb\big(\omega(t)\cap B(x_0,r)\big)\,\mathrm dt.
\end{multline}
In order to control the second integral, we use the dominated convergence theorem which justifies the following convergence
$$\int_0^T\int_{\vert x\vert>r}\big\vert\mathscr F^{-1}_{\xi}(K(T,t,\xi)e^{-l^2\vert\xi\vert^2/2})(x)\big\vert^2\,\mathrm dx\,\mathrm dt \underset{r\rightarrow+\infty}{\rightarrow}0,$$
since
$$\mathscr F^{-1}_{\xi}(K(T,t,\xi)e^{-l^2\vert\xi\vert^2/2})\in L^2([0,T]\times\mathbb R^n).$$
Thus, we can choose the radius $r\gg1$ large enough so that
\begin{equation}\label{19122019E6}
	(2\pi)^nC_{\varepsilon,T}\int_0^T\int_{\vert x\vert>r}\big\vert\mathscr F^{-1}_{\xi}(K(T,t,\xi)e^{-l^2\vert\xi\vert^2/2})(x)\big\vert^2\,\mathrm dx\,\mathrm dt\le \frac{M_l}2.
\end{equation}
Gathering \eqref{16082020E1}, \eqref{23082018E6}, \eqref{23082018E7} and \eqref{19122019E6}, we obtain the following expected estimate
$$\forall x_0\in\mathbb R^n,\quad \frac{M_l}2 \le \frac1{l^{2n}}\int_0^T\Leb\big(\omega(t)\cap B(x_0,r)\big)\,\mathrm dt.$$
This ends the proof of Theorem \ref{19112021T2}.
\end{proof}

\section{Smoothing properties in time and space}\label{smooth}

In order to prove that when that ellipticity condition \eqref{05072021E2} holds for some positive time $T>0$ and that the moving control support $(\omega(t))_{t\in[0,T]}$ satisfies the geometric condition \eqref{23112021E1}, then the evolution equation \eqref{05072021E1} is cost-uniformly approximately null-controllable in time $T$, we need to study the regularity in time and space of the solutions of the adjoint system of \eqref{05072021E1}, which is the retrograde equation given by
$$\left\{\begin{aligned}
	& \partial_tg(t,x) - Q_tD_x\cdot D_xg(t,x) = 0,\quad (t,x)\in[0,T)\times\mathbb R^n, \\
	& g(T,\cdot) = g_0\in L^2(\mathbb R^n).
\end{aligned}\right.$$
Precisely, in this section, still using the notation $U(T,t)$ to denote the Fourier multiplier \eqref{24112021E1}, we prove the following

\begin{thm}\label{06072021T1} Let $T>0$ be a positive time. Assume that the ellipticity condition \eqref{05072021E2} holds and that the matrices $Q_t$ depend analytically on the time variable $t\in\mathbb R$. Then, there exists a positive constant $c_0>1$ not depending on the time $T$ such that for all $m\geq0$, $\alpha\in\mathbb N^n$, $0\le t<T$ and $g\in L^2(\mathbb R^n)$,
\begin{equation}\label{29072021E1}
	\big\Vert\partial^m_t\partial^{\alpha}_x(U(T,t)g)\big\Vert_{L^2(\mathbb R^n)}\le c_0^{m+\vert\alpha\vert}\ \bigg(\frac{\max(1,T)s_T}{c_T(T-t)}\bigg)^{\frac k2(2m+\vert\alpha\vert)}\ m!\ \sqrt{\alpha!}\ \Vert g\Vert_{L^2(\mathbb R^n)},
\end{equation}
where $c_T\in(0,1)$ and $k\geq1$ are the ones appearing in \eqref{05072021E2}, and $s_T>1$ is a positive constant related to the analyticity property of the family $(Q_t)_{t\in\mathbb R}$ on $(-T,T)$.
\end{thm}

The rest of this section is devoted to the proof of this result. 
We first use Plancherel's theorem to get that
\begin{equation}\label{06072021E3}
	\big\Vert\partial^m_t\partial^{\alpha}_x(U(T,t)g)\big\Vert_{L^2(\mathbb R^n)} = (2\pi)^{-n/2}\big\Vert\xi^{\alpha}\partial^m_t(e^{-A_t(\xi)})\widehat g\big\Vert_{L^2(\mathbb R^n)},
\end{equation}
where we set
$$A_t(\xi) = \int_t^TQ_s\xi\cdot\xi\,\mathrm ds.$$
Now, we therefore have to estimate the time-derivatives of the function $\exp\circ(-A_t)$. To that end, we shall use Fa\`a di Bruno's formula in one variable, see e.g. Formula (4.3.2) page 304 in \cite{MR2599384}, whose statement is the following: Given $U,V,W\subset\mathbb R$ some open sets, and $f:U\rightarrow V$, $g:V\rightarrow W$ some smooth functions, we have that for all integer $m\geq0$,
\begin{equation}\label{23112021E3}
	\frac{(g\circ f)^{(m)}}{m!} = \sum_{l_1+2l_2+\cdots+ml_m = m}\frac{g^{(l_1+\cdots+l_m)}\circ f}{l_1!\cdots l_m!}\prod_{j=1}^m\bigg(\frac{f^{(j)}}{j!}\bigg)^{l_j}.
\end{equation}
We get from Fa\`a di Bruno's formula that for all $t>0$, $\xi\in\mathbb R^n$ and $m\geq1$,
\begin{align*}
	\frac{\partial^m_t(e^{-A_t(\xi)})}{m!} & = e^{-A_t(\xi)}\sum_{l_1+2l_2+\cdots+ml_m = m}\frac1{l_1!\cdots l_m!}\prod_{j=1}^m\bigg(\frac{\partial^j_t(-A_t(\xi))}{j!}\bigg)^{l_j} \\[5pt]
	& = e^{-A_t(\xi)}\sum_{l_1+2l_2+\cdots+ml_m = m}\frac1{l_1!\cdots l_m!}\prod_{j=1}^m\bigg(\frac{(\partial^{j-1}_tQ_t)\xi\cdot\xi}{j!}\bigg)^{l_j}.
\end{align*}
Moreover, the matrices $Q_t$ are assumed to depend analytically on the time-variable $t\in\mathbb R$, so there exists a positive constant $s_T>1$ such that for all $t\in(-T,T)$ and $m\geq0$,
$$\big\Vert\partial^m_tQ_t\big\Vert\le s^{1+m}_Tm!,$$
where $\Vert\cdot\Vert$ stands for the norm induced by the canonical Euclidean norm on $\mathbb R^n$. As a consequence of this estimate, we obtain that for all $0<t<T$, $\xi\in\mathbb R^n$ and $m\geq1$, 
\begin{align}\label{23112021E4}
	\bigg\vert\frac{\partial^m_t(e^{-A_t(\xi)})}{m!}\bigg\vert & \le e^{-A_t(\xi)}\sum_{l_1+2l_2+\cdots+ml_m = m}\frac1{l_1!\cdots l_m!}\prod_{j=1}^m\bigg(\frac{s^j_T(j-1)!\vert\xi\vert^2}{j!}\bigg)^{l_j} \\[5pt]
	& = e^{-A_t(\xi)}\sum_{l_1+2l_2+\cdots+ml_m = m}\frac{s^{l_1+2l_2+\cdots+ml_m}_T\vert\xi\vert^{2(l_1+\cdots+l_m)}}{1^{l_1}l_1!\cdots m^{l_m}l_m!} \nonumber \\[5pt]
	& = s^m_Te^{-A_t(\xi)}\sum_{l_1+2l_2+\cdots+ml_m = m}\frac{\vert\xi\vert^{2(l_1+\cdots+l_m)}}{1^{l_1}l_1!\cdots m^{l_m}l_m!}. \nonumber
\end{align}
In order to estimate the above sum, we shall use the following lemma, which is quite a straightforward consequence of Fa\`a di Bruno's formula.

\begin{lem}\label{06072021L1} We have that for all non-negative real number $a\geq0$,
\begin{equation}\label{23112021E2}
	\sum_{l_1+2l_2+\cdots+ml_m = m}\frac{a^{l_1+\cdots+l_m}}{1^{l_1}l_1!\cdots m^{l_m}l_m!} = \frac1{m!}\prod_{j=0}^{m-1}(a+j).
\end{equation}
\end{lem}

\begin{proof} Let $a\geq0$ be a fixed non-negative real number. We consider the function 
$$F:x\in(-1,1) \mapsto -a\ln(1-x).$$ 
The strategy to establish the formula \eqref{23112021E2} is to compute the derivatives of the function $\exp\circ F$ at $0$ with two different methods. On the one hand, we have that this function is given by
$$\forall x\in(-1,1),\quad (\exp\circ F)(x) = \frac1{(1-x)^a} = 1 + \sum_{m=1}^{+\infty}(-1)^m\binom{-a}mx^m,$$
where
$$\binom{-a}m = \frac{-a(-a-1)\cdots(-a-(m-1))}{m!} = \frac{(-1)^m}{m!}\prod_{j=0}^{m-1}(a+j).$$
We therefore deduce that 
$$\forall m\geq1,\quad (\exp\circ F)^{(m)}(0) = \prod_{j=0}^{m-1}(a+j).$$
On the other hand, let us recall that the derivatives of the function $F$ are given by 
$$\forall j\geq1,\quad F^{(j)}(0) = a(j-1)!.$$
As a consequence of this formula and \eqref{23112021E3}, we get that for all $m\geq1$,
\begin{align*}
	\frac{(\exp\circ F)^{(m)}(0)}{m!} & = \sum_{l_1+2l_2+\cdots+ml_m = m}\frac{e^{F(0)}}{l_1!\cdots l_m!}\prod_{j=1}^m\bigg(\frac{F^{(j)}(0)}{j!}\bigg)^{l_j} \\[5pt]
	& = \sum_{l_1+2l_2+\cdots+ml_m = m}\frac1{l_1!\cdots l_m!}\prod_{j=1}^m\bigg(\frac{a(j-1)!}{j!}\bigg)^{l_j} \\[5pt]
	& = \sum_{l_1+2l_2+\cdots+ml_m = m}\frac{a^{l_1+\cdots+l_m}}{1^{l_1}l_1!\cdots m^{l_m}l_m!}.
\end{align*}
This ends the proof of Lemma \ref{06072021L1}.
\end{proof}

Resuming the proof of \eqref{29072021E1}, we deduce from \eqref{23112021E4}, Lemma \ref{06072021L1} and the following classical convexity inequality
$$\forall N\geq1, \forall a,b\geq0, \quad (a+b)^N\le 2^{N-1}(a^N+b^N),$$
that for all $m\geq1$, $0<t<T$ and $\xi\in\mathbb R^n$,
\begin{align*}
	\big\vert\vert\xi\vert^{\vert\alpha\vert}\partial^m_t(e^{-A_t(\xi)})\big\vert
	& \le s^m_T\vert\xi\vert^{\vert\alpha\vert}e^{-A_t(\xi)}\prod_{j=0}^{m-1}(\vert\xi\vert^2+j)
	\le s^m_T\vert\xi\vert^{\vert\alpha\vert}e^{-A_t(\xi)}(\vert\xi\vert^2+m)^m \\[5pt]
	& \le 2^{m-1}s^m_T(\vert\xi\vert^{2m}+m^m)\vert\xi\vert^{\vert\alpha\vert}e^{-A_t(\xi)}.
\end{align*}
There are now two terms to consider. In order to control them, we will use two easy estimates. The first, which comes from a straightforward study of function, states that 
$$\forall p,q>0, \forall x\geq0,\quad x^pe^{-cx^q}\le\bigg(\frac p{ecq}\bigg)^{p/q}.$$
The second one is a consequence of the log-convexity of the function $x\geq0\mapsto x^x$:
$$\forall x,y\geq0,\quad \bigg(\frac{x+y}2\bigg)^{(x+y)/2}\le x^{x/2}y^{y/2}.$$
As a consequence of the ellipticity condition \eqref{05072021E2} and the above two estimates, we therefore have that for all $N\geq0$, $\alpha\in\mathbb N^n$, $0\le t<T$ and $\xi\in\mathbb R^n$,
\begin{align*}
	\vert\xi\vert^{2N+\vert\alpha\vert}e^{-A_t(\xi)} & \le\vert\xi\vert^{2N+\vert\alpha\vert}e^{-c_T(T-t)^k\vert\xi\vert^2}\le\bigg(\frac{2N+\vert\alpha\vert}{2ec_T(T-t)^k}\bigg)^{(2N+\vert\alpha\vert)/2} \\[5pt]
	& \le\frac{(2N)^N\vert\alpha\vert^{\vert\alpha\vert/2}}{(ec_T(T-t)^k)^{(2N+\vert\alpha\vert)/2}},
\end{align*}
where $c_T\in(0,1)$ is the constant appearing in \eqref{05072021E2}. Moreover, since $m^m\le e^mm!$ and $\vert\alpha\vert!\le n^{\vert\alpha\vert}\alpha!$ (consequence of the definition of the exponential function and generalized Newton's formula), we obtain that for all $m\geq0$, $\alpha\in\mathbb N^n$, $0\le t<T$ and $\xi\in\mathbb R^n$,
\begin{align*}
	\big\vert\vert\xi\vert^{\vert\alpha\vert}\partial^m_t(e^{-A_t(\xi)})\big\vert 
	& \le 2^{m-1}s^m_T\frac{(2m)^m\vert\alpha\vert^{\vert\alpha\vert/2}}{(ec_T(T-t)^k)^{(2m+\vert\alpha\vert)/2}} + 2^{m-1}s^m_T\frac{m^m\vert\alpha\vert^{\vert\alpha\vert/2}}{(ec_T(T-t)^k)^{\vert\alpha\vert/2}} \\[5pt]
	& \le 2^{m-1}s^m_T\frac{(2e)^mm!(en)^{\vert\alpha\vert/2}\sqrt{\alpha!}}{(ec_T(T-t)^k)^{(2m+\vert\alpha\vert)/2}} + 2^{m-1}s^m_T\frac{e^mm!(en)^{\vert\alpha\vert/2}\sqrt{\alpha!}}{(ec_T(T-t)^k)^{\vert\alpha\vert/2}}. 
\end{align*}
There are now two cases to consider. First, when $0<T-t\le1$, we obtain the following bound
$$\big\vert\vert\xi\vert^{\vert\alpha\vert}\partial^m_t(e^{-A_t(\xi)})\big\vert \le \frac{c_1^{m+\vert\alpha\vert}s^m_T}{c_T^{(2m+\vert\alpha\vert)/2}(T-t)^{\frac k2(2m+\vert\alpha\vert)}}\ m!\ \sqrt{\alpha!},$$
where $c_1>0$ is a positive constant only depending on $c, c_0, e$ and $n$. In the other case where $T-t>1$, we use the fact that 
$$\frac1{(T-t)^{\frac k2\vert\alpha\vert}}\le 1 = \frac{(T-t)^{\frac k2(2m+\vert\alpha\vert)}}{(T-t)^{\frac k2(2m+\vert\alpha\vert)}}\le\frac{T^{\frac k2(2m+\vert\alpha\vert)}}{(T-t)^{\frac k2(2m+\vert\alpha\vert)}},$$
which implies that
$$\big\vert\vert\xi\vert^{\vert\alpha\vert}\partial^m_t(e^{-A_t(\xi)})\big\vert \le c_2^{m+\vert\alpha\vert} \frac{\max(1,T)^{\frac k2(2m+\vert\alpha\vert)}s^m_T}{c_T^{(2m+\vert\alpha\vert)/2}(T-t)^{\frac k2(2m+\vert\alpha\vert)}}\ m!\ \sqrt{\alpha!},$$
where $c_2>0$ is another positive constant only depending on $c, c_0, e$ and $n$. These two estimates on $\vert\xi\vert^{\vert\alpha\vert}\partial^m_t(e^{-A_t(\xi)})$ combined with \eqref{06072021E3}, Plancherel's theorem and the fact that $c_T\in(0,1)$, $s_T>1$ imply that \eqref{29072021E1} holds. This ends the proof of Theorem \ref{06072021T1}.

\section{Sufficient condition for approximate null-controllability with uniform cost}\label{suffcond}

This section is devoted to the proof of the direct implication in Theorem \ref{19112021T1}. 
Anew, we keep using the notation $U(T,t)$ to denote the Fourier multiplier \eqref{24112021E1}. Precisely, we aim at establishing the following quantitative unique continuation property:

\begin{thm}\label{17092021T1} Let $T>0$ be a positive time and $(\omega(t))_{t\in[0,T]}$ be a moving control support satisfying the following integral thickness condition:
\begin{equation}\label{mean_thick}
	\exists\gamma\in(0,1], \exists r>0, \forall x \in\mathbb R^n, \quad \frac1T\int_0^T \Leb\big(\omega(t) \cap B(x,r)\big)\, \mathrm dt \geq \gamma V_r.
\end{equation}
When the ellipticity condition \eqref{05072021E2} holds and that the matrices $Q_t$ depend analytically on the time variable $t\in\mathbb R$, there exist some positive constants $C_n>0$ and $K_n>0$ only depending on the dimension $n$ such that for all $0<\varepsilon\le\varepsilon_0$ and $g\in L^2(\mathbb R^n)$,
$$\int_0^{T_{\gamma}}\big\Vert U(T,t)g\big\Vert^2_{L^2(\mathbb R^n)}\ \mathrm dt \le \bigg(\frac{K_n(2-\gamma)}{\gamma}\bigg)^{K_nC}\int_0^{T_{\gamma}} \big\Vert U(T,t)g\big\Vert^2_{L^2(\omega(t))}\, \mathrm dt+ \varepsilon \Vert g\Vert^2_{L^2(\mathbb R^n)},$$
where we set
\begin{equation}\label{17122021E1}
	T_{\gamma} = \Big(1-\frac{\gamma}2\Big)T\quad\text{and}\quad\varepsilon_0 = C_n\sqrt{\frac{T_{\gamma}V_1}{r^n}},
\end{equation}
and where the constant $C = C_{\varepsilon,\gamma,r,k,T}>0$ is given by
\begin{equation}\label{17122021E2}
	C = \bigg(1-\log(\varepsilon r^n) +\log \bigg(1+\frac{C^{2k}_T}{\gamma^{2k}T^{2(k-1)}}+ \frac{r^2C^k_T}{\gamma^kT^k} \bigg)\bigg)\exp\bigg(\frac{K_nC^k_T}{\gamma^kT^{k-1}}\bigg)+ \frac{r^2C^k_T}{\gamma^kT^k},
\end{equation}
with $C_T = \max(1,T)s_T/c_T$, $c_T\in(0,1)$ and $k\geq1$ being the ones involved in \eqref{05072021E2}, and $s_T>1$ being a positive constant related to the analyticity property of the family $(Q_t)_{t\in\mathbb R}$ on $(-T,T)$.
\end{thm}

Before proving Theorem \ref{17092021T1}, let us check that the direct implication in Theorem \ref{19112021T1} is a consequence of this result. Recall from Corollary \ref{12012022C1} that proving that the equation \eqref{05072021E1} is cost-uniformly approximately null-controllable from any moving control support satisfying the geometric condition \eqref{mean_thick}, when the ellipticity condition \eqref{05072021E2} holds, is equivalent to obtaining a weak observability estimate of the following form for all $\varepsilon\in(0,1)$ and $g\in L^2(\mathbb R^n)$,
$$\big\Vert U(T,0)g\big\Vert^2_{L^2(\mathbb R^n)}\le C_{\varepsilon,T}\int_0^T\big\Vert U(T,t)g\big\Vert^2_{L^2(\omega(t))}\, \mathrm dt + \varepsilon\Vert g\Vert^2_{L^2(\mathbb R^n)}.$$
In fact, such an inequality is an immediate consequence of Theorem \ref{17092021T1}, since the norm $\Vert U(T,t)g\Vert_{L^2(\mathbb R^n)}$ is increasing with respect to $t$, which provides that
\begin{align*}
	\big\Vert U(T,0)g\big\Vert^2_{L^2(\mathbb R^n)} & = \frac1{T_{\gamma}}\int_0^{T_{\gamma}}\big\Vert U(T,0)g\big\Vert^2_{L^2(\mathbb R^n)}\,\mathrm dt \\[5pt]
	& \le \frac1{T_{\gamma}}\int_0^{T_{\gamma}}\big\Vert U(T,t)g\big\Vert^2_{L^2(\mathbb R^n)}\,\mathrm dt \\[5pt]
	& \le \frac1{T_{\gamma}}\bigg(\frac{K_n(2-\gamma)}{\gamma}\bigg)^{K_nC}\int_0^{T_{\gamma}} \big\Vert U(T,t)g\big\Vert^2_{L^2(\omega(t))}\, \mathrm dt+ \frac{\varepsilon}{T_{\gamma}}\Vert g\Vert^2_{L^2(\mathbb R^n)}.
\end{align*}

Instrumental in the proof of Theorem \ref{17092021T1} are the following quantitative unique continuation estimates, whose proof is postponed in the Subsection \ref{spectral} of the appendix.

\begin{prop}\label{30112021P1} Let $A,B \geq 1$ be positive constants, $n\in\mathbb N^*$ be a dimension, $0< t \leq 1$ be a rate, $0< s < 1$ be a positive real number and $0< \gamma \leq 1$ be another rate. We also consider $E \subset (-1,1) \times B(0,1) \subset \mathbb R^{n+1}$ a measurable subset such that $\Leb E\geq 2\gamma V_1$. Then, there exists a constant $K_{s,n} \geq 1$ such that for all $f \in \mathcal C^{\infty}((-1,1) \times B(0,1))$ satisfying 
\begin{equation}\label{30112021E2}
	\Vert f \Vert_{L^{\infty}((-1,1)\times B(0,1))} \geq t,
\end{equation}
and
\begin{equation}\label{24112021E2}
	\forall m \in\mathbb N, \forall\alpha\in\mathbb N^n, \quad \Vert\partial_u^m\partial^{\alpha}_xf\Vert_{L^{\infty}((-1,1)\times B(0,1))} \le A^m B^{\vert\alpha\vert}\, m!\, (\vert\alpha\vert!)^s,
\end{equation}
the following estimate holds
\begin{equation}\label{30112021E1}
	\Vert f\Vert^2_{L^2((-1,1)\times B(0,1))} \leq \bigg(\frac{K_{s,n}}{\gamma}\bigg)^{K_{s,n}((1-\log t) e^{K_{s,n}A}+ B^{\frac1{1-s}})} \Vert f\Vert^2_{L^2(E)}.
\end{equation}
\end{prop}

We can now tackle the proof of Theorem \ref{17092021T1}, which is divided in five steps.

\medskip

\noindent\textit{$\triangleright$ Step 1: A thick set in time and space.} The first step consists in claiming that the condition \eqref{mean_thick} is a thickness condition for a subset of $[0,T]\times\mathbb R^n$. Before checking this fact, we notice that
\begin{equation}\label{10112021E5}
	\forall x\in\mathbb R^n, \quad\int_0^{T_{\gamma}}\Leb\big(\omega(t) \cap B(x,r)\big)\,\mathrm dt \geq\frac{\gamma}2TV_r,
\end{equation}
where we set $T_{\gamma} = (1-\gamma/2)T$. Indeed, we deduce from \eqref{mean_thick} that for all $x\in\mathbb R^n$,
\begin{align*}
	\gamma TV_r & \leq \int_0^T \Leb\big(\omega(t) \cap B(x,r)\big)\,\mathrm dt \\
	& \leq \int_0^{T_{\gamma}} \Leb\big(\omega(t) \cap B(x,r)\big)\,\mathrm dt+\int_{T_{\gamma}}^T \Leb\big(\omega(t) \cap B(x,r)\big)\,\mathrm dt \\
	& \leq \int_0^{T_{\gamma}} \Leb\big(\omega(t) \cap B(x,r)\big)\,\mathrm dt+\frac{\gamma}2 TV_r.
\end{align*}
By now, we consider the following measurable subset of $[0,T_{\gamma}]\times\mathbb R^n$
$$\Omega= \Big\{(t,x)\in[0,T_{\gamma}]\times\mathbb R^n : x \in \omega(t)\Big\}.$$
As a consequence of \eqref{10112021E5}, we deduce that for all $x\in\mathbb R^n$,
\begin{equation}\label{mean_thick1}
	\Leb\big(\Omega \cap[0,T_{\gamma}] \times B(x,r)\big) = \int_0^{T_{\gamma}} \Leb\big(\omega(t) \cap B(x,r)\big)\,\mathrm dt
	\geq \frac{\gamma}{2}T V_r.
\end{equation}
This proves that $\Omega$ is a thick subset of $[0,T_{\gamma}]\times\mathbb R^n$.

\medskip

\noindent\textit{$\triangleright$ Step 2: Definition of good and bad cylinders.} The remaining of the proof consists in using elements of harmonic analysis. In order to divide the set $[0,T_{\gamma}]\times\mathbb R^n$ into cylinders, we notice that since the ellipticity condition \eqref{05072021E2} holds and the matrices $Q_t$ depend analytically on the time variable $t\in\mathbb R$, we get the following Bernstein estimates from Theorem~\ref{06072021T1}: there exists a positive constant $c_0>1$ not depending on $T$ such that such that for all $m\geq0$, $\alpha\in\mathbb N^n$ and $g\in L^2(\mathbb R^n)$,
\begin{equation}\label{bernstein_estimate}
	\big\Vert\partial^m_t\partial^{\alpha}_x(U(T,\cdot)g)\big\Vert_{L^2([0,T_{\gamma}]\times\mathbb R^n)}\le \sqrt{T_{\gamma}}\,c_0^{m+\vert\alpha\vert}\,\bigg(\frac{2C_T}{\gamma T}\bigg)^{\frac k2(2m+\vert\alpha\vert)}\ m!\ \sqrt{\alpha!}\ \Vert g\Vert_{L^2(\mathbb R^n)},
\end{equation}
where we set
$$C_T = \frac{\max(1,T)s_T}{c_T}.$$
Indeed, it follows from Theorem~\ref{06072021T1} that we have for all $m\geq0$, $\alpha\in\mathbb N^n$, $0\le t<T$ and $g\in L^2(\mathbb R^n)$,
$$\big\Vert\partial^m_t\partial^{\alpha}_x(U(T,t)g)\big\Vert_{L^2(\mathbb R^n)}\le c_0^{m+\vert\alpha\vert}\ \bigg(\frac{\max(1,T)s_T}{c_T(T-t)}\bigg)^{\frac k2(2m+\vert\alpha\vert)}\ m!\ \sqrt{\alpha!}\ \Vert g\Vert_{L^2(\mathbb R^n)}.$$
By integrating in time, and using that
$$\int_0^{T_{\gamma}}\frac{\mathrm dt}{(T-t)^{k(2m+\vert\alpha\vert)}}\le\int_0^{T_{\gamma}}\frac{\mathrm dt}{(T-T_{\gamma})^{k(2m+\vert\alpha\vert)}}
= T_{\gamma}\,\frac{2^{k(2m+\vert\alpha\vert)}}{(\gamma T)^{k(2m+\vert\alpha\vert)}},$$
we therefore deduce that the Bernstein estimates \eqref{bernstein_estimate} hold. 

For $\beta\in r\mathbb Z^n$, let us now define the cylinder $C(\beta)$ by
$$C(\beta) =[0,T_{\gamma}] \times B(\beta, r).$$
Notice that the family $(C(\beta))_{\beta\in r\mathbb Z^n}$ covers the set $[0,T_{\gamma}]\times\mathbb R^n$:
\begin{equation}\label{06112020E5}
	[0,T_{\gamma}]\times\mathbb R^n = \bigcup_{\beta\in r\mathbb Z^n}C(\beta),
\end{equation}
and also satisfies the following intersection property:
$$\forall (\beta_1,\ldots, \beta_{10})\in r\mathbb Z^n\ \text{such that}\ \beta_k\ne\beta_l\ \text{when}\ 1\leq k\neq l\leq10,\quad\bigcap_{k=1}^{10} C(\beta_k)=\emptyset.$$
As a consequence, we have
\begin{equation}\label{recov1}
	\forall x\in [0,T_{\gamma}]\times\mathbb R^n,\quad 1\le\sum\limits_{\beta\in r\mathbb Z}\mathbbm1_{C(\beta)}(x)\le9.
\end{equation}
For the remaining of this proof, we fix $g \in L^2(\mathbb R^n)$ and $\varepsilon>0$. A cylinder $C(\beta)$ is said to be good if it satisfies that for all $m\in\mathbb N$ and $\alpha\in\mathbb N^n$,
\begin{equation}\label{good}
	\big\Vert \partial_t^m\partial^{\alpha}_x(U(T,\cdot)g)\big\Vert_{L^2(C(\beta))}\le \frac{3\sqrt{2T_{\gamma}}}{\sqrt{\varepsilon}}\,c_0^{m+\vert\alpha\vert}\,\bigg(\frac{4C_T}{\gamma T}\bigg)^{\frac k2(2m+\vert\alpha\vert)}\, m!\, \sqrt{\vert\alpha\vert!}\, \big\Vert U(T,\cdot)g\big\Vert_{L^2(C(\beta))}.
\end{equation}
Naturally, a cylinder $C(\beta)$ is said to be bad if it is not good, that is, when there exist a non-negative integer $m_0\in\mathbb N$ and a multiindex $\alpha_0\in\mathbb N^n$ such that
\begin{multline}\label{bad}
	\big\Vert\partial_t^{m_0}\partial^{\alpha_0}_x(U(T,\cdot)g)\big\Vert_{L^2(C(\beta))}>\frac{3\sqrt{2T_{\gamma}}}{\sqrt{\varepsilon}}\,c_0^{m_0+\vert\alpha_0\vert}\,\bigg(\frac{4C_T}{\gamma T}\bigg)^{\frac k2(2m_0+\vert\alpha_0\vert)}\, \\[5pt]
	\times m_0!\, \sqrt{\vert\alpha_0\vert!}\, \big\Vert U(T,\cdot)g\big\Vert_{L^2(C(\beta))}.
\end{multline}
Notice from the covering property \eqref{06112020E5} that
\begin{equation}\label{gb1}
	\big\Vert U(T,\cdot)g\big\Vert^2_{L^2([0,T_{\gamma}]\times\mathbb R^n)} \le \sum_{\text{g.c.}} \big\Vert U(T,\cdot)g\big\Vert^2_{L^2(C(\beta))} + \sum_{\text{b.c.}}\big\Vert U(T,\cdot)g\big\Vert^2_{L^2(C(\beta))},
\end{equation}
where g.c. stands for ``good cylinders'' and b.c. stands for ``bad cylinders''.

\medskip

\noindent\textit{$\triangleright$ Step 3: Estimates for the bad cylinders.} We shall estimate independently the two terms in the right-hand side of the inequality \eqref{gb1}. Let us begin with the second one.
It follows from the definition \eqref{bad} that if $C(\beta)$ is a bad cylinder, there exist a non-negative integer $m_0\in\mathbb N$ and a multiindex $\alpha_0\in\mathbb N^n$ such that
\begin{align}\label{bad2}
	\big\Vert U(T, \cdot) g\big\Vert^2_{L^2(C(\beta))} & \le\frac{\varepsilon(\gamma T)^{k(2m_0+|\alpha_0|)}}{18T_{\gamma}c_0^{2(m_0+\vert\alpha_0\vert)}(4C_T)^{k(2m_0 + \vert\alpha_0\vert)} (m_0!)^2\vert\alpha_0\vert!}\, \big\Vert\partial_t^{m_0}\partial^{\alpha_0}_x(U(T,\cdot)g)\big\Vert^2_{L^2(C(\beta))} \\[5pt]
	& \le\sum_{m\in \mathbb N,\,\alpha\in\mathbb N^n}\frac{\varepsilon (\gamma T)^{k(2m+|\alpha|)}}{18T_{\gamma}c_0^{2(m+\vert\alpha\vert)}(4C_T)^{k(2m + \vert\alpha\vert)} (m!)^2 \vert\alpha\vert!}\, \big\Vert\partial_t^m\partial^{\alpha}_x(U(T,\cdot)g)\big\Vert^2_{L^2(C(\beta))}. \nonumber
\end{align}
By summing over all the bad cylinders and using the fact that $\alpha!\le\vert\alpha\vert!$, we obtain from the Bernstein estimate \eqref{bernstein_estimate}, the covering property \eqref{recov1} and \eqref{bad2} that 
\begin{align}\label{12042021E2}
	&\ \sum_{\text{b.c.}}\big\Vert U(T,\cdot)g\big\Vert^2_{L^2(C(\beta))} \\ \nonumber
	\le &\ \varepsilon\sum_{\text{b.c.}}\sum_{m \in \mathbb N,\,\alpha\in\mathbb N^n}\frac{(\gamma T)^{k(2m+\vert\alpha\vert)}}{18T_{\gamma}c_0^{2(m+\vert\alpha\vert)}(4C_T)^{k(2m + \vert\alpha\vert)}(m!)^2\vert\alpha\vert!}\,\big\Vert\partial_t^m\partial^{\alpha}_x(U(T,\cdot)g)\big\Vert^2_{L^2(C(\beta))} \\[5pt]
	\le &\ \frac\varepsilon2\sum_{m\in\mathbb N,\,\alpha\in\mathbb N^n}\frac {(\gamma T)^{k(2m+\vert\alpha\vert)}}{T_{\gamma}c_0^{2(m+\vert\alpha\vert)}(4C_T)^{k(2m + \vert\alpha\vert)}(m!)^2\vert\alpha\vert!}\,\big\Vert\partial_t^m\partial^{\alpha}_x(U(T,\cdot)g)\big\Vert^2_{L^2([0,T_{\gamma}]\times\mathbb R^n)} \nonumber \\[5pt]
	\le &\ \frac\varepsilon2\sum_{m\in\mathbb N,\,\alpha\in\mathbb N^n}\frac1{4^{2m + \vert\alpha\vert}}\,\Vert g\Vert^2_{L^2(\mathbb R^n)} 
	\le \varepsilon\Vert g \Vert^2_{L^2(\mathbb R^n)}. \nonumber
\end{align}

\smallskip

\noindent\textit{$\triangleright$ Step 4: Estimates for the good cylinders.} It remains to estimate the first term in the right-hand side of the inequality \eqref{gb1}. To that end, we will use Proposition \ref{30112021P1}. This step is the most technical part of the paper. In order to alleviate the writing, we denote by $C_n$ a positive constant depending only on the dimension $n$, and whose value may change from a line to another. Let $C(\beta)$ be a good cylinder and $C = (-1,1)\times B(0,1)$. As a first step, we establish that there exists a positive constant $C_n>0$ such that for all $m\geq0$ and $\alpha\in\mathbb N^n$,
\begin{multline}\label{goodinf}
	\big\Vert \partial_t^m\partial^{\alpha}_x(U(T,T_{\gamma}(1+ \cdot)/2)g(\beta+r\,\cdot))\big\Vert^2_{L^{\infty}(C)}
	\le \frac1{\varepsilon r^n}
	\bigg(1+\frac{C^{2k}_T}{\gamma^{2k}T^{2(k-1)}}+\frac{r^2C^k_T}{\gamma^kT^k}\bigg)^{n+1} \\[5pt]
	\times C^{1+2m+\vert\alpha\vert}_n\,\bigg(\frac{C^k_T}{\gamma^kT^{k-1}}\bigg)^{2m}\,\bigg(\frac{r^2C^k_T}{\gamma^kT^k}\bigg)^{\vert\alpha\vert}\,(m!)^2\,\vert\alpha\vert!\,\big\Vert U(T,\cdot)g\big\Vert^2_{L^2(C(\beta))}.
\end{multline}
To that end, we begin by noticing that since the cylinder $C=(-1,1) \times B(0,1)$ satisfies the cone condition, the following Sobolev embedding holds, see e.g.~\cite[Theorem~4.12]{adams},
$$W^{n+1,2}(C) \hookrightarrow L^{\infty}(C).$$
This implies that there exists a positive constant $C_n>0$ such that 
$$\forall u \in W^{n+1,2}(C), \quad \Vert u\Vert_{L^{\infty}(C)}\le C_n\Vert u\Vert_{W^{n+1,2}(C)}.$$
It follows from this estimate and a change of variable that for all $m\geq0$ and $\alpha\in\mathbb N^n$,
\begin{align*}
	&\ \big\Vert \partial_t^m\partial^{\alpha}_x(U(T,T_{\gamma}(1+ \cdot)/2)g(\beta+r\,\cdot))\big\Vert^2_{L^{\infty}(C)} \\[5pt]
	\le &\ C^2_n \sum_{\tilde m+\vert\tilde\alpha\vert \le n+1} \big\Vert \partial_t^{m+\tilde m}\partial^{\alpha+\tilde{\alpha}}_x(U(T,T_{\gamma}(1+\,\cdot)/2)g(\beta+r\,\cdot))\big\Vert^2_{L^2(C)} \\[5pt]
	= &\ \frac{2C^2_n}{T_{\gamma}r^n}\sum_{\tilde m+\vert\tilde\alpha\vert \le n+1}T_{\gamma}^{2(m+\tilde m)} r^{2\vert\alpha+\tilde\alpha\vert}\big\Vert\partial_t^{m+\tilde m}\partial^{\alpha+\tilde\alpha}_x(U(T, \cdot)g)\big\Vert^2_{L^2(C(\beta))} \\[5pt]
	\le &\ \frac{2C^2_n}{T_{\gamma}r^n}\sum_{\tilde m+\vert\tilde\alpha\vert \le n+1}T^{2(m+\tilde m)}r^{2\vert\alpha+\tilde\alpha\vert}\big\Vert\partial_t^{m+\tilde m}\partial^{\alpha+\tilde\alpha}_x(U(T, \cdot)g)\big\Vert^2_{L^2(C(\beta))},
\end{align*}
where the the sums are taken over $\tilde m\in\mathbb N$ and $\tilde{\alpha}\in\mathbb N^n$. By using the definition \eqref{good} of good cube, we deduce that there exists a new constant $C_n>0$ such that for all $m\geq0$ and $\alpha\in\mathbb N^n$,
\begin{multline*}
	\big\Vert \partial_t^m\partial^{\alpha}_x(U(T,T_{\gamma}(1+ \cdot)/2)g(\beta+r\,\cdot))\big\Vert^2_{L^{\infty}(C)}
	\le\frac{2C^2_n}{\varepsilon T_{\gamma}r^n}\sum_{\tilde m+\vert\tilde\alpha\vert \le n+1}18T_{\gamma}\,T^{2(m+\tilde m)}r^{2\vert\alpha+\tilde\alpha\vert} \\
	\times c_0^{2(m+\tilde m+\vert\alpha+\tilde\alpha\vert)}\,\bigg(\frac{4C_T}{\gamma T}\bigg)^{k(2(m+\tilde m)+\vert\alpha+\tilde\alpha\vert)}\, ((m+\tilde m)!)^2\,\vert\alpha+\tilde\alpha\vert!\, \big\Vert U(T,\cdot)g\big\Vert^2_{L^2(C(\beta))}.
\end{multline*}
Using the fact that when $\tilde m+\vert\tilde\alpha\vert \le n+1$,
$$(m+\tilde m)!\,\vert\alpha+\tilde\alpha\vert!\le2^{m + \tilde m + \vert\alpha+\tilde\alpha\vert}\,m!\,\tilde m!\,\vert\alpha\vert!\,\vert\tilde\alpha\vert!\le4^{n+1}((n+1)!)^2\,2^{m + \vert\alpha\vert}\,m!\,\vert\alpha\vert!,$$
we deduce that there exists a new positive constant $C_n>0$ such that the above estimates rewrite in the following way
\begin{multline*}
	\big\Vert \partial_t^m\partial^{\alpha}_x(U(T,T_{\gamma}(1+ \cdot)/2)g(\beta+r\,\cdot))\big\Vert^2_{L^{\infty}(C)}
	\le \frac1{\varepsilon r^n}
	\sum_{\tilde m+\vert\tilde\alpha\vert \le n+1}\bigg(\frac{C^k_T}{\gamma^kT^{k-1}}\bigg)^{2\tilde m}\,\bigg(\frac{r^2C^k_T}{\gamma^kT^k}\bigg)^{\vert\tilde\alpha\vert} \\[5pt]
	\times C^{1+2m+\vert\alpha\vert}_n\,\bigg(\frac{C^k_T}{\gamma^kT^{k-1}}\bigg)^{2m}\,\bigg(\frac{r^2C^k_T}{\gamma^kT^k}\bigg)^{\vert\alpha\vert}\,(m!)^2\,\vert\alpha\vert!\,\big\Vert U(T,\cdot)g\big\Vert^2_{L^2(C(\beta))}.
\end{multline*}
Moreover, the sum can be estimated as follows
$$\sum_{\tilde m+\vert\tilde\alpha\vert \le n+1}\bigg(\frac{C^k_T}{\gamma^kT^{k-1}}\bigg)^{2\tilde m}\,\bigg(\frac{r^2C^k_T}{\gamma^kT^k}\bigg)^{\vert\tilde\alpha\vert}
\le C_n\bigg(1+\frac{C^{2k}_T}{\gamma^{2k}T^{2(k-1)}}+\frac{r^2C^k_T}{\gamma^kT^k}\bigg)^{n+1}.$$
This proves that the estimate \eqref{goodinf} actually holds. Assuming that the function $U(T, \cdot)g$ is not identically equal to zero on the cylinder $C(\beta)$, we define the function $\varphi: (-1,1)\times B(0,1) \rightarrow\mathbb C$ for all $(u,z) \in (-1,1)\times B(0,1)$ by
\begin{equation}\label{function_aux}
	\varphi(u, z)= \frac{\varepsilon r^n(U(T, T_{\gamma}(1+u)/2)g)(\beta +r z)}{C_n\Big(1+\frac{C^{2k}_T}{\gamma^{2k}T^{2(k-1)}}+\frac{r^2C^k_T}{\gamma^kT^k}\Big)^{n+1}\Vert U(T, \cdot)g\Vert_{L^2(C(\beta))}}.
\end{equation} 
It follows from \eqref{goodinf} that the function $\varphi$ satisfies the following estimates for all $m\in\mathbb N$ and $\alpha\in\mathbb N^n$,
\begin{equation}\label{17092021E1}
	\big\Vert \partial_u^m\partial^{\alpha}_z\varphi\big\Vert^2_{L^{\infty}((-1,1) \times B(0,1))}
	 \le \bigg(\frac{C_nC^k_T}{\gamma^kT^{k-1}}\bigg)^{2m}\,\bigg(\frac{r^2C_nC^k_T}{\gamma^kT^k}\bigg)^{\vert\alpha\vert}\,(m!)^2\,\vert\alpha\vert!.
\end{equation}
Moreover, the $L^{\infty}$-norm of the function $\varphi$ is also bounded from below as follows
\begin{align}\label{15122021E1}
	\Vert\varphi\Vert_{L^{\infty}((-1,1)\times B(0,1))} & = \frac{\varepsilon r^n\Vert U(T,\cdot)g\Vert_{L^{\infty}(C(\beta))}}{C_n\Big(1+\frac{C^{2k}_T}{\gamma^{2k}T^{2(k-1)}}+\frac{r^2C^k_T}{\gamma^kT^k}\Big)^{n+1}\Vert U(T, \cdot)g\Vert_{L^2(C(\beta))}} \\
	& \geq\frac{\varepsilon r^n}{C_n\Big(1+\frac{C^{2k}_T}{\gamma^{2k}T^{2(k-1)}}+\frac{r^2C^k_T}{\gamma^kT^k}\Big)^{n+1}\sqrt{\Leb C(\beta)}} =: t. \nonumber
\end{align}
Notice that considering $\varepsilon_0 = \varepsilon_{0,n,\gamma,r,T}>0$ defined in \eqref{17122021E1}, we get that $0<t\le 1$ provided $0<\varepsilon\le\varepsilon_0$. This is due to the fact that by definition of $t$,
$$0<t = \frac{\varepsilon r^n}{C_n\Big(1+\frac{C^{2k}_T}{\gamma^{2k}T^{2(k-1)}}+\frac{r^2C^k_T}{\gamma^kT^k}\Big)^{n+1}\sqrt{\Leb C(\beta)}}\le \frac{\varepsilon r^{n/2}}{C_n\sqrt{T_{\gamma}V_1}}.$$
Let us now define the following measurable set 
$$E = \big\{(u,z)\in(-1,1)\times B(0,1) : (T_{\gamma}(1+u)/2, \beta + rz) \in \Omega\big\}.$$
We deduce from \eqref{mean_thick1} that the measure of $E$ satisfies
\begin{equation}\label{15122021E2}
	\Leb E = \frac{\Leb\big(\Omega \cap C(\beta)\big)}{(T_{\gamma}/2)r^n} \geq \frac{(\gamma/2)TV_r}{(T_{\gamma}/2)r^n} = \frac{\gamma}{2-\gamma}\, 2V_1,\quad\text{with}\quad \frac{\gamma}{2-\gamma}\in(0,1].
\end{equation}
As a consequence of \eqref{17092021E1}, \eqref{15122021E2} and Proposition \ref{30112021P1} applied to the function $\varphi$, there exists a positive constant $K_n\geq1$ only depending on the dimension $n$, such that 
$$\Vert\varphi\Vert^2_{L^2((-1,1)\times B(0,1))} \leq \bigg(\frac{K_n(2-\gamma)}{\gamma}\bigg)^{K_n((1-\log t) e^{K_nA}+ B^2)} \Vert\varphi\Vert^2_{L^2(E)},$$
where $0<t\le 1$ is the one appearing in \eqref{15122021E1}, and where we set 
$$A = \frac{C_nC^k_T}{\gamma^kT^{k-1}}\quad\text{and}\quad B = \bigg(\frac{r^2C_nC^k_T}{\gamma^kT^k}\bigg)^{1/2}.$$
Up to slightly modifying the positive constant $K_n$, the above estimate can be rewritten in the following form
\begin{equation}\label{15122021E3}
	\Vert\varphi\Vert^2_{L^2((-1,1)\times B(0,1))}\le \bigg(\frac{K_n(2-\gamma)}{\gamma}\bigg)^{K_nC} \Vert\varphi\Vert^2_{L^2(E)},
\end{equation}
where the positive constant $C = C_{\varepsilon,\gamma,r,k,T}>0$ is given by \eqref{17122021E2}. By changing variables, it directly follows from the definition \eqref{function_aux} of the function $\varphi$ and the estimate \eqref{15122021E3} that
\begin{equation}\label{quasi_analytic1}
	\big\Vert U(T,\cdot)g\big\Vert^2_{L^2(C(\beta))}
	\le \bigg(\frac{K_n(2-\gamma)}{\gamma}\bigg)^{K_nC}\big\Vert U(T,\cdot)g\big\Vert^2_{L^2(\Omega \cap C(\beta))}.
\end{equation}
This inequality also holds when the function $U(T, \cdot)g$ is identically equal to zero on the cylinder $C(\beta)$. By summing over all the good cylinders, we therefore deduce from \eqref{recov1} and \eqref{quasi_analytic1} that
\begin{align}\label{15122021E5}
	\sum_{\textrm{g.c.}} \big\Vert U(T, \cdot)g \big\Vert^2_{L^2(C(\beta))} & \le 9\bigg(\frac{K_n(2-\gamma)}{\gamma}\bigg)^{K_nC}\big\Vert U(T,\cdot)g\big\Vert^2_{L^2(\Omega\cap\cup_{\text{g.c.}} C(\beta))} \\[5pt]
	& \le 9\bigg(\frac{K_n(2-\gamma)}{\gamma}\bigg)^{K_nC}\big\Vert U(T,\cdot)g\big\Vert^2_{L^2(\Omega)}. \nonumber
\end{align}

\medskip

\noindent\textit{$\triangleright$ Step 5: End of the proof.} Gathering the estimates \eqref{gb1}, \eqref{12042021E2} and \eqref{15122021E5}, and slightly modifying the constant $K_n$, we obtain that for all $g\in L^2(\mathbb R^n)$,
$$\big\Vert U(T,\cdot)g\big\Vert^2_{L^2([0,T_{\gamma}]\times\mathbb R^n)}\le\bigg(\frac{K_n(2-\gamma)}{\gamma}\bigg)^{K_nC} \int_0^{T_{\gamma}} \big\Vert U(T,t)g\big\Vert^2_{L^2(\omega(t))}\, \mathrm dt+ \varepsilon\Vert g\Vert^2_{L^2(\mathbb R^n)}.$$
This is the expected estimate. The proof of Theorem \ref{17092021T1} is therefore now ended.

\section{Appendix}

\subsection{Weak observability} Let us begin this appendix by stating the cost-uniform approximate null-controllability of the equation \eqref{05072021E1} in term of a weak observability estimate. In this subsection, we keep using the notation $U(T,t)$ to denote the Fourier multiplier \eqref{24112021E1}. Moreover, there is no particular assumption on the family $(Q_t)_{t\in\mathbb R}$. The following result is an adaptation of Lemma 3.4 in the work \cite{MR2679651}, and its proof is given for the sake of completeness of the present paper.

\begin{prop}\label{03092020P1} Given the time $T>0$, the cost $C>0$, the approximation rate $\varepsilon>0$ and the moving control support $(\omega(t))_{t\in[0,T]}$, the two following properties
\begin{multline*}
	\forall f_0\in L^2(\mathbb R^n),\exists h\in L^2((0,T)\times\mathbb R^n), \\
	\frac1C\int_0^T\Vert h(t,\cdot)\Vert^2_{L^2(\omega(t))}\,\mathrm dt
	+\frac1{\varepsilon}\Vert f(T,\cdot)\Vert^2_{L^2(\mathbb R^n)}\le \Vert f_0\Vert^2_{L^2(\mathbb R^n)},
\end{multline*}
where $f$ stands for the mild solution of the equation \eqref{05072021E1} with initial datum $f_0$ and control $h$, and
$$\forall g\in L^2(\mathbb R^n),\quad \big\Vert U(T,0)g\big\Vert^2_{L^2(\mathbb R^n)}\le C\int_0^T\big\Vert U(T,t)g\big\Vert^2_{L^2(\omega(t))}\, \mathrm dt + \varepsilon\Vert g\Vert^2_{L^2(\mathbb R^n)},$$
are equivalent.
\end{prop}

In view of the definition of approximate null-controllability with uniform cost stated in the beginning of Section \ref{results}, we deduce the following

\begin{cor}\label{12012022C1} Let $T>0$ be a positive time and $(\omega(t))_{t\in[0,T]}$ be a moving control support. The equation \eqref{05072021E1} is cost-uniformly approximately null-controllable in time $T$ from $(\omega(t))_{t\in[0,T]}$ if and only if for all $\varepsilon\in(0,1)$, there exists a positive constant $C_{\varepsilon,T}>0$ such that for all $g\in L^2(\mathbb R^n)$,
$$\big\Vert U(T,0)g\big\Vert^2_{L^2(\mathbb R^n)}\le C_{\varepsilon,T}\int_0^T\big\Vert U(T,t)g\big\Vert^2_{L^2(\omega(t))}\, \mathrm dt + \varepsilon\Vert g\Vert^2_{L^2(\mathbb R^n)}.$$
\end{cor}

\begin{proof}[Proof of Proposition \ref{03092020P1}] Consider $T>0$, $C>0$, $\varepsilon>0$ and a moving control support $(\omega(t))_{t\in[0,T]}$. We first assume that for all $f_0\in L^2(\mathbb R^n)$ there exists a control $h\in L^2((0,T)\times\mathbb R^n)$ such that
\begin{equation}\label{cuanc}
	\frac1C\int_0^T\Vert h(t,\cdot))\Vert^2_{L^2(\omega(t))}\, \mathrm dt + \frac1{\varepsilon} \Vert f(T,\cdot)\Vert^2_{L^2(\mathbb R^n)} \le \Vert f_0\Vert^2_{L^2(\mathbb R^n)}.
\end{equation}
Notice that the function $f(T,\cdot)$ is given (by definition) by 
$$f(T,\cdot)= U(T,0)f_0 + \int_0^TU(T,t)(h(t,\cdot)\mathbbm1_{\omega(t)})\, \mathrm dt.$$
Let $g \in L^2(\mathbb R^n)$. We deduce from the selfadjointness of the operators $U(T,t)$ and the above equality that for all $f_0\in L^2(\mathbb R^n)$,
\begin{align*}
	\langle f_0,U(T,0)g \rangle_{L^2(\mathbb R^n)} & = \langle U(T,0)f_0,  g \rangle_{L^2(\mathbb R^n)} \\
	& = \langle f(T,\cdot),g\rangle_{L^2(\mathbb R^n)} - \int_0^T\big\langle U(T,t)(h(t,\cdot)\mathbbm1_{\omega(t)}),g\big\rangle_{L^2(\mathbb R^n)}\, \mathrm dt \\
	& = \langle f(T,\cdot),g\rangle_{L^2(\mathbb R^n)} - \int_0^T\big\langle h(t,\cdot),U(T,t)g\big\rangle_{L^2(\omega(t))}\, \mathrm dt.
\end{align*}
Cauchy-Schwarz' inequality in the space $L^2((0,T)\times\mathbb R^n)\times L^2(\mathbb R^n)$ then implies that for all $f_0\in L^2(\mathbb R^n)$,
\begin{multline*}
	\big\vert\langle f_0,U(T,0)g \rangle_{L^2(\mathbb R^n)}\big\vert^2
	\le \bigg(\frac1{\varepsilon}\Vert f(T,\cdot)\Vert^2_{L^2(\mathbb R^n)} + \frac1C\int_0^T \Vert h(t,\cdot)\Vert^2_{L^2(\omega(t))}\,\mathrm dt\bigg) \\
	\times\bigg(\varepsilon\Vert g\Vert^2_{L^2(\mathbb R^n)} + C\int_0^T\big\Vert U(T,t)g\big\Vert^2_{L^2(\omega(t))}\, \mathrm dt\bigg).
\end{multline*}
By using the estimate \eqref{cuanc} and by choosing $f_0= U(T,0)g$, we therefore obtain the following weak observability estimate for all $g\in L^2(\mathbb R^n)$,
\begin{equation}\label{observability_estimate}
	\big\Vert U(T,0)g\big\Vert^2_{L^2(\mathbb R^n)} \le C \int_0^T \big\Vert U(T,t)g\big\Vert^2_{L^2(\omega(t))}\, \mathrm dt +\varepsilon \Vert g\Vert^2_{L^2(\mathbb R^n)}.
\end{equation}

Conversely, let us assume that the weak observability estimate \eqref{observability_estimate} holds for all $g \in L^2(\mathbb R^n)$. Considering a fixed $f_0 \in L^2(\mathbb R^n)$, we consider the following $C^1$ convex functional $J:L^2(\mathbb R^n)\rightarrow\mathbb R$ defined for all $f\in L^2(\mathbb R^n)$ by 
$$J(f)= \frac C2 \int_0^T \big\Vert U(T,t)f\big\Vert^2_{L^2(\omega(t))}\, \mathrm dt + \frac{\varepsilon}2 \Vert f\Vert^2_{L^2(\mathbb R^n)} + \big\langle U(T,0)f, f_0\big\rangle_{L^2(\mathbb R^n)}.$$
The functional $J$ is immediately coercive since we have from Cauchy-Schwarz' inequality that for all $f \in L^2(\mathbb R^n)$,
$$J(f) \ge \frac{\varepsilon}2 \Vert f\Vert^2_{L^2(\mathbb R^n)} - \Vert f\Vert_{L^2(\mathbb R^n)}\Vert f_0\Vert_{L^2(\mathbb R^n)}.$$
As a consequence, there exists a function $h_0 \in L^2(\mathbb R^n)$ such that 
$$J(h_0) = \min_{f \in L^2(\mathbb{R}^n)} J(f).$$ 
In particular, we have 
\begin{equation*}
	\nabla J( h_0) = C \int_0^TU(T,t)(\mathbbm1_{\omega(t)}U(T,t)h_0)\, \mathrm dt +\varepsilon h_0+ U(T,0)f_0=0.
\end{equation*}
It follows from the above equality that the mild solution $f$ of the equation \eqref{05072021E1} with the control
$$h(t,\cdot)= CU(T,t)h_0,$$
satisfies
\begin{align*}
	f(T,\cdot) & = U(T,0)f_0 + \int_0^TU(T,t)(h(t,\cdot)\mathbbm1_{\omega(t)})\, \mathrm dt \\
	& = U(T,0)f_0 + C\int_0^TU(T,t)(\mathbbm1_{\omega(t)}U(T,t)h_0)\, \mathrm dt = -\varepsilon h_0.
\end{align*}
On the other hand, we have
\begin{multline}\label{obs01}
	\langle\nabla J(h_0), h_0\rangle_{L^2(\mathbb R^n)} = C \int_0^T\Vert U(T,t)h_0\Vert^2_{L^2(\omega(t))}\, \mathrm dt \\
	+ \varepsilon\Vert h_0\Vert^2_{L^2(\mathbb R^n)} + \langle h_0,U(T,0)f_0 \rangle_{L^2(\mathbb R^n)} = 0.
\end{multline}
This implies that 
\begin{align*}
	\frac1C\int_0^T\Vert h(t,\cdot)\Vert^2_{L^2(\omega(t))}\, \mathrm dt+ \frac1{\varepsilon} \Vert f(T,\cdot)\Vert^2_{L^2(\mathbb R^n)} 
	& = -\big\langle U(T,0)h_0, f_0\big\rangle_{L^2(\mathbb R^n)} \\[5pt]
	& \leq \Vert U(T,0)h_0\Vert_{L^2(\mathbb R^n)}\Vert f_0\Vert_{L^2(\mathbb R^n)}.
\end{align*}
It only remains to estimate the term $\Vert U(T,0)h_0\Vert_{L^2(\mathbb R^n)}$. We deduce from the weak observability estimate \eqref{observability_estimate}, the equality \eqref{obs01} and Cauchy-Schwarz' inequality that
\begin{align*}
	\Vert U(T,0)h_0\Vert^2_{L^2(\mathbb R^n)} & \le C \int_0^T \big\Vert U(T,t)h_0\big\Vert^2_{L^2(\omega(t))}\, \mathrm dt +\varepsilon \Vert h_0\Vert^2_{L^2(\mathbb R^n)} \\[5pt]
	& = -\langle U(T,0)h_0, f_0 \rangle_{L^2(\mathbb R^n)} \le \Vert U(T,0)h_0\Vert_{L^2(\mathbb R^n)}\Vert f_0\Vert_{L^2(\mathbb R^n)}.
\end{align*}
We therefore deduce that
$$\Vert U(T,0)h_0\Vert_{L^2(\mathbb R^n)} \le \Vert f_0\Vert_{L^2(\mathbb R^n)}.$$
This ends the proof of Proposition \ref{03092020P1}.
\end{proof}

\subsection{Unique continuation}\label{spectral} In this second subsection, we give the proof of Proposition \ref{30112021P1}, which was key in the proof of Theorem \ref{17092021T1} in Section \ref{suffcond}. To that end, we will rely on the following multidimensional version of a theorem by Nazarov-Sodin-Volberg \cite[Theorem~B]{NSV}, proven by the second author in the work \cite{M}.

\begin{prop}{\cite[Example~5.11]{M}}\label{17092021P1} Let $A\geq 1$ be a positive constant, $R>0$ be a radius, $d\geq1$ be a dimension, $ 0<t\leq1$ be a rate, $0<s\leq1$ be a positive real number and $0< \gamma\leq1$ be another rate. We also consider $E \subset B(0,R) \subset \mathbb R^d$ a measurable set such that $\Leb E \geq \gamma V_R$. Then, there exists a constant $C_{s,d, A,R, t} \geq 1$ such that for all $f \in \mathcal C^{\infty}(B(0,R))$ satisfying 
$$\Vert f\Vert_{L^{\infty}(B(0,R))} \geq t,$$ 
and
$$\forall\alpha\in\mathbb N^n, \quad \Vert\partial_x^{\alpha} f\Vert_{L^{\infty}(B(0,R))} \le A^{\vert\alpha\vert}\,(\vert\alpha\vert!)^s,$$
the following estimate holds
$$\Vert f\Vert^2_{L^{\infty}(B(0,R))} \le C_{s,d,A,R,t}\,\Vert f\Vert^2_{L^{\infty}(E)}.$$
Moreover: 
\begin{itemize}[label=\textbf.,leftmargin=* ,parsep=2pt,itemsep=0pt,topsep=2pt]
	\item[.] When $0<s<1$, there exists a constant $K_{s,d}\geq1$, only depending on $s$ and $d$, such that
	$$C_{s,d,A,R,t} \le \bigg(\frac{K_{s,d}}{\gamma}\bigg)^{K_{s,d}(1-\log t+ (AR)^{\frac1{1-s}})}.$$
	\item[.] When $s=1$, there exists a constant $K_d\geq1$, only depending on $d$, such that
$$C_{1,d,A,R,t}\le\bigg(\frac{K_d}{\gamma}\bigg)^{K_d(1-\log t)e^{K_dRA}}.$$
\end{itemize}
\end{prop}

Let us now begin the proof of Proposition \ref{30112021P1}. In order to establish the estimate \eqref{30112021E1}, we follow the strategy implemented by B. Jaye and M. Mitkovski in the work \cite{JM}. Before getting into the heart of the proof, notice that the assumption \eqref{24112021E2} implies that the function $f$ and all its derivatives are Lipschitz, so the estimates \eqref{30112021E2} and \eqref{24112021E2} can be extended on the compact set $[-1,1]\times \overline{B(0,1)}$.

\medskip

\noindent\textit{$\triangleright$ Step 1: Unique continuation in time.} The first step consists in applying Proposition \ref{17092021P1} with respect to the time variable. Precisely, we will apply this result to the function $u\in I\mapsto f(u,x_0)$, where $x_0\in\overline{B(0,1)}$ will be chosen in a while, and the set $I\subset(-1,1)$ is defined by
$$I = \Big\{u \in (-1,1) : \Leb E_u\geq\frac{\gamma}2V_1\Big\},$$
where the sets $E_u$ are given for all $u\in(-1,1)$ by
$$E_u = \big\{x \in B(0,1) : (u,x) \in E \big\}.$$
We first notice that $\Leb I \geq \gamma$. Indeed, we deduce from the assumption on $E$ that
\begin{align*}
	2\gamma V_1\le\Leb E
	& = \int_{-1}^1 \Leb E_u\, \mathrm du \\[5pt]
	& = \int_I   \Leb E_u\,\mathrm du + \int_{(-1,1) \setminus I} \Leb E_u\, \mathrm du
	\leq(\Leb I+\gamma)V_1.
\end{align*}
Moreover, the function $f$ being continuous, we can now consider $(u_0, x_0) \in [-1,1]\times \overline{B(0,1)}$ such that 
$$\vert f(u_0,x_0)\vert = \Vert f\Vert_{L^{\infty}([-1,1]\times\overline{B(0,1)})}.$$
Noticing from \eqref{30112021E2} and \eqref{24112021E2} respectively that the function $u\in I\mapsto f(u,x_0)$ satisfies
$$\Vert f(\cdot,x_0)\Vert_{L^{\infty}(-1,1)}\geq\vert f(u_0,x_0)\vert = \Vert f\Vert_{L^{\infty}([-1,1]\times\overline{B(0,1)})}\geq t,$$
and
$$\forall m\geq0,\quad \Vert\partial_u^mf(\cdot,x_0)\Vert_{L^{\infty}(-1,1)} \le \Vert\partial_u^mf\Vert_{L^{\infty}((-1,1)\times B(0,1))} \le A^m\, m!,$$
we deduce from Proposition~\ref{17092021P1} that there exists a positive constant $K_1\geq 1$ such that 
\begin{equation}\label{nsv_time}
	\Vert f(\cdot,x_0)\Vert_{L^{\infty}(-1,1)} \le \bigg(\frac{K_1}{\gamma}\bigg)^{K_1(1-\log t)e^{K_1A}}\Vert f(\cdot,x_0)\Vert_{L^{\infty}(I)}.
\end{equation}

\medskip

\noindent\textit{$\triangleright$ Step 2: Unique continuation in space.} The second step consists in applying Proposition \ref{17092021P1} to the function $x\in B(0,1)\mapsto f(u_1,x)$, where $u_1\in I$ is defined in the following way: fixing $\varepsilon >0$, we consider $u_1\in I$ satisfying
$$\vert f(u_1,x_0)\vert\geq\Vert f(\cdot,x_0)\Vert_{L^{\infty}(I)} - \varepsilon.$$
On the one hand, we get from the above inequality and \eqref{nsv_time} that
\begin{equation}\label{30112021E4}
	\Vert f(u_1,\cdot) \Vert_{L^{\infty}(B(0,1))} \geq \vert f(u_1,x_0)\vert \geq \Vert f(\cdot,x_0)\Vert_{L^{\infty}(I)} - \varepsilon \geq t_{\varepsilon},
\end{equation}
where we set
$$t_{\varepsilon} = \bigg(\frac{K_1}{\gamma}\bigg)^{-K_1(1-\log t)e^{K_1A}}t - \varepsilon.$$
We assume that $0<\varepsilon\ll1$ is small enough so that $0<t_{\varepsilon}\le 1$. On the other hand, it follows from \eqref{24112021E2} that the function $f(u_1,\cdot)$ enjoys the following regularity
$$\forall\alpha\in\mathbb N^n, \quad\Vert\partial^{\alpha}_xf(u_1,\cdot)\Vert_{L^{\infty}(B(0,1))} \le \Vert\partial_u^m\partial^{\alpha}_xf\Vert_{L^{\infty}((-1,1)\times B(0,1))} \le B^{\vert\alpha\vert}\, (\vert\alpha\vert!)^s.$$
Moreover, since $u_1 \in I$, we have $\Leb E_{u_1} \geq \gamma V_1/2$ by definition of the set $I$, and Proposition~\ref{17092021P1} gives the existence of a positive constant $K_n\geq 1$ such that
\begin{align}\label{30112021E3}
	\Vert f(u_1,\cdot)\Vert_{L^{\infty}(B(0,1))} & \le \bigg(\frac{K_n}{\gamma}\bigg)^{K_n(1-\log t_{\varepsilon} + B^{\frac1{1-s}})}\,\Vert f(u_1,\cdot)\Vert_{L^{\infty}(E_{u_1})} \\
	& \le \bigg(\frac{K_n}{\gamma}\bigg)^{K_n(1-\log t_{\varepsilon} + B^{\frac1{1-s}})}\,\Vert f\Vert_{L^{\infty}(E)}. \nonumber
\end{align}

\medskip

\noindent\textit{$\triangleright$ Step 3: Unique continuation in time and space.} We now gather the two estimates established in the two first steps. We deduce from \eqref{nsv_time}, \eqref{30112021E4} and \eqref{30112021E3} that 
\begin{align*}
	\Vert f\Vert_{L^{\infty}((-1,1)\times B(0,1))} & \leq \bigg(\frac{K_1}{\gamma}\bigg)^{K_1(1-\log t)e^{K_1A}}\Vert f(\cdot,x_0)\Vert_{L^{\infty}(I)} \\
	& \le \bigg(\frac{K_1}{\gamma}\bigg)^{K_1(1-\log t)e^{K_1A}}(\Vert f(u_1,\cdot)\Vert_{L^{\infty}(B(0,1))} + \varepsilon) \\
	& \le \bigg(\frac{K_1}{\gamma}\bigg)^{K_1(1-\log t)e^{K_1A}}\bigg(\bigg(\frac{K_n}{\gamma}\bigg)^{K_n(1-\log t_{\varepsilon} + B^{\frac1{1-s}})}\Vert f\Vert_{L^{\infty}(E)} + \varepsilon\bigg).
\end{align*}
By letting $\varepsilon$ tend to $0^+$ and noticing that
\begin{align*}
	0\le 1-\log t_0 & = 1 - \log\bigg(\bigg(\frac{K_1}{\gamma}\bigg)^{-K_1(1-\log t)e^{K_1A}}t\bigg) \\[5pt]
	& = 1 - \log t + K_1(1-\log t)e^{K_1A}\log\bigg(\frac{K_1}{\gamma}\bigg) \\[5pt]
	& \le (1 - \log t)e^{K_1A} + K_1(\log K_1)(1-\log t)e^{K_1A},
\end{align*}
it follows that there exists a new positive constant $K_{1,n}\geq1$ such that
\begin{equation}\label{06102021nsv}
	\Vert f \Vert_{L^{\infty}((-1,1)\times B(0,1))}\le\bigg(\frac{K_{1,n}}{\gamma}\bigg)^{K_{1,n}((1-\log t)e^{K_{1,n}A}+ B^{\frac1{1-s}})}\,\Vert f\Vert_{L^{\infty}(E)}.
\end{equation} 

\medskip

\noindent\textit{$\triangleright$ Step 4: From the $L^{\infty}$-norm to the $L^2$-norm.} In this last step, we check that the $L^{\infty}$-norm can be replaced by the $L^2$-norm in the estimate \eqref{06102021nsv}. To that end, we consider
$$\tilde E= \bigg\{(u,x)\in E : |f(u,x)| \le \frac2{\Leb E}\int_E\vert f\vert\bigg\}.$$
It follows from the definition of $\tilde E$ that
$$\frac{2\Leb(E \setminus \tilde E)}{\Leb E}\int_E\vert f\vert\le\int_{E\setminus\tilde E}\vert f\vert\le\int_E\vert f\vert.$$
If $\int_E\vert f\vert\neq 0$, we deduce that 
$$\frac{2\Leb(E\setminus\tilde E)}{\Leb E}\le1,$$
and as a consequence,
$$\Leb\tilde E\geq\frac{\Leb E}2.$$
In the case where $\int_E\vert f\vert = 0$, we deduce that $\Leb \tilde E = \Leb E$ and the above estimate holds as well. It follows that \eqref{06102021nsv} also holds when $E$ is replaced by $\tilde E$ and $\gamma$ is replaced by $\gamma/2$. We therefore obtain that
$$\Vert f\Vert_{L^{\infty}((-1,1)\times B(0,1))} \le \Big(\frac{2K_{1,n}}{\gamma}\Big)^{K_{1,n}((1-\log t)e^{K_{1,n}A}+ B^{\frac1{1-s}})}\Vert f\Vert_{L^{\infty}(\tilde E)}.$$
As a consequence of Cauchy-Schwarz' inequality, the $L^2$-norm of the function $f$ is bounded in the following way
\begin{align*}
	\Vert f\Vert_{L^2((-1,1)\times B(0,1))} & \le \Leb((-1,1)\times B(0,1))^{1/2} \Vert f\Vert_{L^{\infty}((-1,1)\times B(0,1))} \\
	& \le \Leb((-1,1)\times B(0,1))^{1/2}\bigg(\frac{2K_{1,n}}{\gamma}\bigg)^{K_{1,n}((1-\log t)e^{K_{1,n}A} + B^{\frac1{1-s}})}\Vert f\Vert_{L^{\infty}(\tilde E)} \\
	& \le \frac{2\Leb((-1,1)\times B(0,1))^{1/2}}{\Leb E}\bigg(\frac{2K_{1,n}}{\gamma}\bigg)^{K_{1,n}((1-\log t)e^{K_{1,n}A}+ B^{\frac1{1-s}})}\int_E\vert f\vert \\
	& \le \frac{2\Leb((-1,1)\times B(0,1))^{1/2}}{(\Leb E)^{1/2}}\bigg(\frac{2K_{1,n}}{\gamma}\bigg)^{K_{1,n}((1-\log t)e^{K_{1,n}A}+ B^{\frac1{1-s}})}\Vert f\Vert_{L^2(E)}.
\end{align*}
Moreover, it follows from the assumption $\Leb E\geq 2\gamma V_1$ that
$$\frac{2\Leb((-1,1)\times B(0,1))^{1/2}}{(\Leb E)^{1/2}} = \frac{2(2V_1)^{1/2}}{(\Leb E)^{1/2}}\le\frac2{\gamma^{1/2}}\le\frac2{\gamma}.$$
Therefore, by slightly modifying the constant $K_{1,n}$, we obtain that
$$\Vert f\Vert_{L^2((-1,1)\times B(0,1))} \le \bigg(\frac{K_{1,n}}{\gamma}\bigg)^{K_{1,n}((1-\log t)e^{K_{1,n}A}+ B^{\frac1{1-s}})}\Vert f\Vert_{L^2(E)}.$$
This ends the proof of Proposition \ref{30112021P1}.

\end{document}